\documentclass{amsart}
\usepackage{amssymb,latexsym,amsmath,amsthm}
\usepackage[mathscr]{eucal}
\newtheorem{theorem}{Theorem}[section]
\newtheorem{lemma}[theorem]{Lemma}

\theoremstyle{definition}
\newtheorem{remark}[theorem]{Remark}

\numberwithin{equation}{section}

\renewcommand{\Re}{\text{Re}\,}
\renewcommand{\Im}{\text{Im}\,}

\newcommand{\RR}{\ensuremath{\mathbb{R}}}
\newcommand{\CC}{\ensuremath{\mathbb{C}}}

\newcommand{\prtl}{\ensuremath{\partial}}

\newcommand{\supp}{\ensuremath{\text{supp}}}
\newcommand{\veps}{\ensuremath{\varepsilon}}

\newcommand{\nablatx}{\ensuremath{\nabla_{t,x}}}

\title[Estimates for the Wave Equation in Strictly Concave Domains]{Strichartz and Localized Energy Estimates for the Wave Equation in Strictly Concave Domains}
\thanks{The author was supported in part by the National Science Foundation, grants DMS-1301717 and DMS-1001529.}

\author[M. D. Blair]{Matthew D. Blair}
\address{Department of Mathematics and Statistics, University of New Mexico, Albuquerque, NM 87131, USA}
\email{blair@math.unm.edu}

\begin{document}
\begin{abstract}
We prove localized energy estimates for the wave equation in domains with a strictly concave boundary when homogeneous Dirichlet or Neumann conditions are imposed.  By restricting the solution to small, frequency dependent, space time collars of the boundary, it is seen that a stronger gain in regularity can be obtained relative to the usual energy estimates.  Mixed norm estimates of Strichartz and square function type follow as a result, using the energy estimates to control error terms which arise in a wave packet parametrix construction.  While the latter estimates are not new for Dirichlet conditions, the present approach provides an avenue for treating these estimates when Neumann conditions are imposed.  The method also treats Schr\"odinger equations with time independent coefficients. 
\end{abstract}
\maketitle

\section{Introduction}
Let $(M,g)$ be a $C^\infty$ Riemannian manifold of dimension $n\geq 2$ with compact, strictly geodesically concave boundary and $\Delta_g$ the \emph{nonnegative} Dirichlet or Neumann Laplacian.  Suppose $u,v: [-1,1] \times M \to \CC$ are solutions to the wave and Schr\"odinger equations
\begin{align}
(D_t^2-\Delta_g)u(t,x)&=0, &(u,\prtl_t u)|_{t=0}&=(f,g),\label{waveeqn}\\
(D_t+\Delta_g)v(t,x)&=0, &v|_{t=0} &= f,\label{schrodeqn}
\end{align}
(with $D_t = -i\prtl_t$) subject to either homogeneous Dirichlet or Neumann boundary conditions
\begin{equation}\label{BCs}
u(t,x)|_{x\in \prtl M} = 0 \qquad \text{ or } \qquad N u(t,x)|_{x\in \prtl M} = 0 ,
\end{equation}
where $N$ is a normal vector field along $\prtl \Omega$ and $u$ can be replaced by $v$ as appropriate.  We are concerned with establishing local Strichartz estimates for these solutions with $p,q>2$
\begin{align}
\|u\|_{L^p([-1,1]; L^q( M)) }& \leq C \left(\|f\|_{H^{\gamma}( M)} + \|g\|_{H^{\gamma-1}( M)}\right), & \frac 2p + \frac{n-1}q &\leq \frac{n-1}2, \label{wavestz}\\
\|v\|_{L^p([-1,1]; L^q( M)) }& \leq C \|f\|_{H^{\alpha}( M)}, & \frac 2p + \frac nq &\leq \frac n2.\label{schrodstz}
\end{align}
Here and in what follows, $H^s( M)$ denotes the $L^2$ Sobolev space of order $\gamma$ determined by the functional calculus of the Dirichlet/Neumann operator (see for example, the concluding comments in \cite[\S1]{bssschrod}). In particular, we are interested in \emph{scale invariant} estimates, ones where the regularity on the right hand side is the optimal regularity predicted by scaling, meaning that
\begin{align}
\frac 1p + \frac nq &= \frac n2-\gamma, & \text{ (Wave)}, \label{wavescale}\\
\frac 2p + \frac nq &= \frac n2-\alpha, &  \text{ (Schr\"odinger)}.\label{schrodscale}
\end{align}
For scale invariant triples $(p,q,\gamma)$ or $(p,q,\alpha)$, the Knapp example, a solution to the respective equations which is highly concentrated along a light ray, imposes the additional restrictions on the right in \eqref{wavestz}, \eqref{schrodstz}.  For the Schr\"odinger equation, this simply means that $\alpha \geq0$.  When strict inequality appears in \eqref{wavestz}, \eqref{schrodstz}, the estimate does not use the full rate of dispersion suggested by boundaryless problem $ M = \RR^n$, and hence such scale invariant triples are said to be \emph{subcritical}.  When $ M = \RR^n$ it is well known that the full range of Strichartz estimates are satisfied, at least up to certain endpoint cases which we neglect here.

The imposition of the boundary conditions \eqref{BCs} affect the flow of energy, complicating the development of these estimates for boundary value problems considerably, as one must now account for the geometry of the boundary and the nature of the condition itself (Dirichlet or Neumann).  In the present set of hypotheses, that $\prtl  M$ is strictly concave and compact, the full range of local estimates \eqref{wavestz} for the wave equation are known when Dirichlet conditions are imposed, this is due to Smith and Sogge \cite{SmSoCrit}.  When $ M$ is exterior to a strictly convex obstacle, the same is true for the Schr\"odinger equation \eqref{schrodstz}. This is due to Ivanovici \cite{ivanomt}, who also showed corresponding scale invariant estimates for the semiclassical equation with frequency localized data.  In these works, the authors used the Melrose-Taylor diffractive parametrix which is effective for homogeneous Dirichlet conditions given further developments of Zworski \cite{zworski}.

On the other hand, the author with Smith and Sogge proved a family of local, scale invariant estimates for the wave and Schr\"odinger equations \cite{bssbdrystz}, \cite{bssschrod}, operating under relaxed set of hypotheses on the geometry of the boundary and on the conditions themselves.  Namely, in the case of the wave equation, the boundary merely needs to be smooth and compact, while for Schr\"odinger equations, $ M$ need only be a nontrapping exterior domain (though scale invariant estimates of semiclassical type are proved in the process).  Moreover, both types of boundary conditions are allowed.  However, the admissibility conditions on $p,q$ are restrictive in comparison to the estimates in free space as for wave equations it is assumed that $\frac 3p + \frac{n-1}{q} \leq \frac{n-1}{2}$ when $n \leq 4$ and $\frac 1p +\frac 1q \leq \frac 12$ when $n\geq 4$ and for Schr\"odinger equations admissibility is defined by replacing $n-1$ by $n$ here.  These works were based on a parametrix construction due to Smith and Sogge \cite{smithsogge07} which gave similar estimates of square function type on arbitrary manifolds with boundary.

A common thread in the proof of Schr\"odinger estimates referenced above is the use of a local smoothing bound of Burq, G\'erard, and Tzvetkov \cite{bgtexterior} satisfied for solutions in nontrapping exterior domains
\begin{equation}\label{locsmoothunit}
\|\psi v\|_{L^2([-T,T];H^{s+\frac 12}( M))} \lesssim \|f\|_{H^{s}( M)}, \qquad \psi \in C^\infty_c(\overline{ M}).
\end{equation}
The interplay between these bounds and Strichartz estimates is observed in works of Journ\'e, Soffer, and Sogge \cite{jss} and Staffilani and Tataru \cite{ST}, and the arguments in the latter translate particularly well to boundary value problems.  The role of local smoothing bounds is to show that one can control the errors which arise by localizing the solution in the spatial variables, reducing matters to proving estimates on solutions which are concentrated near a frequency scale $\lambda$ and within a coordinate chart.  This allows one to employ a local parametrix construction within a chart without conflicting with the infinite propagation speed of solutions. In particular, after taking a Littlewood-Paley decomposition of the solution and passing to weighted $L^2$ estimates, matters are reduced to showing estimates for solutions to a ``semiclassical" Schr\"odinger equation over a unit time scale such as $[-1,1]$
\begin{equation}\label{semiclasseqn}
(D_t + \lambda^{-1} \Delta_g) v_\lambda =0, \qquad v_\lambda(t,\cdot)\text{ spectrally localized to frequencies $\approx \lambda$}.
\end{equation}
Since solutions to this equation exhibit propagation speed which uniformly bounded in $\lambda$, Strichartz estimates can be treated by a local parametrix construction.  Further details on these localization arguments can be found in \cite[\S2]{bssschrod}, \cite[\S2.1]{blairrls}, \cite[\S3]{ivanomt}, and elsewhere.  In the case of the wave equation, infinite speed of propagation is no longer an obstacle to local estimates. The analogous program is thus to combine the classical local energy decay estimates in nontrapping exterior domains with local Strichartz estimates near the boundary to yield Strichartz estimates over global time scales. This is due to Smith and Sogge \cite{SmSoGlobal}, Metcalfe \cite{metcalfeglobal}, and Burq \cite{burqglobal}.

In \cite{blairrls}, the author furthered these ideas, proving that whenever a refined family of estimates for the semiclassical Schr\"odinger equation in strictly concave domains are satisfied, scale invariant Strichartz estimates follow as a result. To state them, given a small dyadic number $\lambda^{-\frac 23} \leq 2^{-j} \leq 1$, we define
\begin{equation}\label{sjdef}
S_{<j} = \{(t,x) \in  [-1,1] \times M: d(x,\prtl  M) \leq 2^{-j} \}.
\end{equation}
The refined local smoothing estimates thus assert that one obtains an additional gain of $2^{- j/4}$ in the usual local smoothing estimates by restricting the solution to such a small collar of the boundary.  The heuristic argument for such a bound is that a wave packet concentrated along a glancing ray should spend a time comparable to $2^{-j/2}$ within $S_{<j}$, so that the square integral in time yields a gain of the square root of that amount.
\begin{theorem}\label{thm:blairrls}
Suppose  $\prtl M$ is strictly concave and compact, and that $v_\lambda$ is an arbitrary solution to the \textbf{semiclassical} Schr\"odinger equation \eqref{semiclasseqn}, satisfying \eqref{BCs}, with spectral localization taken with respect to the functional calculus of the respective Dirichlet or Neumann Laplacian $v_\lambda(t,\cdot) = \beta(\lambda^{-2}\Delta_g)v_\lambda(t,\cdot)$, $\beta \in C_c^\infty(\lambda/2,2\lambda)$.
\begin{enumerate}
\item If $M$ is compact and $v_\lambda(0,\cdot) \in L^2( M)$, then the refined local smoothing estimates
\begin{equation}\label{rlsschrod}
2^{\frac j4}\|v_\lambda\|_{L^2( S_{<j})} \leq C\|v_\lambda(0,\cdot)\|_{L^2( M)},
\end{equation}
are satisfied for any $\lambda^{-\frac 23} \leq 2^{-j} \leq 1$ with $C$ independent of $\lambda$, $j$, $v_\lambda$.
\item If \eqref{rlsschrod} is satisfied with $C$ uniform in $\lambda,j,v_\lambda$, then solutions to \eqref{semiclasseqn} with $v_\lambda(0,\cdot) \in L^2(M)$ satisfy the semiclassical Strichartz estimate
\[
\|v_\lambda\|_{L^p([-1,1]; L^q( M))} \leq C\lambda^{\alpha+\frac 1p}\|v_\lambda(0,\cdot)\|_{L^2( M)}, \qquad \text{provided } \frac 2p + \frac nq < \frac n2.
\]
Consequently, if $ M$ is the exterior of a connected, strictly convex obstacle the scale invariant estimates for the classical equation in \eqref{schrodstz} hold for the same exponents.  If $ M$ is compact, then \eqref{schrodstz} holds for the subcritical exponents, but with a loss of $1/p$ derivatives, that is, $H^\alpha$ on the right hand side is replaced by $H^{\alpha+1/p}$. \label{rlsconsequence}
\end{enumerate}
\end{theorem}
When $M$ is a domain in $\RR^n$ exterior to a ball, the bound \eqref{rlsschrod} is due to Ivanovici \cite{ivanoballs}.  When $\prtl M$ is not assumed to  have this structure, these bounds are new and will be seen as a consequence of Theorem \ref{thm:intrinsic} below in \S\ref{ss:schrod}.  As alluded to above, \eqref{rlsconsequence} is the main theorem in \cite{blairrls}.  Strictly speaking, the estimates assumed there are for the classical Schr\"odinger equation \eqref{schrodeqn}, and take the form \eqref{rlsschrodclassical} below, but as noted in \cite[\S2.2]{blairrls}, the estimates for the semiclassical equation which follow are the crucial element in the proof.  The consequences for exterior problem follow from the aforementioned localization arguments and the consequences for compact manifolds follow from the methods for boundaryless manifolds developed by Burq, G\'erard, and Tzvetkov \cite{bgtmanifold}.  In short, the crucial idea in \cite{blairrls} is that the refined estimates \eqref{rlsschrod} in some sense allow one to control the errors which arise through certain microlocalizations of the solution, and therefore estimates can be obtained by the aforementioned wave packet parametrix in \cite{bssschrod}.  A similar program was carried out for asymptotically flat Schr\"odinger and wave equations by Tataru \cite{tataruschrod} and Metcalfe and Tataru \cite{metcalfetataru}, developing a family of refined local smoothing bounds for such equations, then showing that these bounds can be using to estimate the error terms arising from a wave packet parametrix construction.

The purpose of the present work is twofold, one is to prove \eqref{rlsschrod} as a consequence of similar bounds for the wave equation, thus verifying the crucial hypothesis in \cite{blairrls}.  The second is to prove theorems for the wave equation analogous to the second half of Theorem \ref{thm:blairrls}, ones which will yield \eqref{wavestz} for scale invariant triples which are subcritical with respect to the Knapp example $\frac 2p+\frac{n-1}{q} < \frac{n-1}{2}$.  The arguments apply equally well to the subcritical square function estimates, which replace the Strichartz norm by $L^q( M;L^2_t([-1,1]))$ and yield $L^q$ bounds on spectral clusters/quasimodes determined by $\Delta_g$ for compact manifolds.

A significant degree of progress is therefore made in establishing the family of estimates surrounding the Stein-Tomas restriction theorem to the setting of strictly concave boundaries when Neumann conditions are imposed.    It is expected that the full range of these local estimates should hold for Neumann conditions, and that the restriction to subcritical triples is an artifact of the methods presented here.  Our results also apply to Dirichlet conditions, but the full range of estimates are already known in this setting.

It is not known if the scale invariant Strichartz and square function estimates can be obtained using the Melrose-Taylor diffractive parametrix for homogeneous Neumann conditions.  The main obstacle seems to be that the so-called ``Neumann operator", which amounts to the inverse of the Dirichlet to Neumann map, is smoothing of order 2/3 instead of order 1, see the manuscript of Melrose and Taylor \cite[Ch. 8]{mt}.  An interesting feature of the present work is that the proof of the localized energy estimates under consideration rely on regularity estimates for boundary traces due to Tataru \cite{tatarubdry}. They not only show that the Dirichlet traces associated to the Neumann problem are smoother on the order of 5/6 instead of 2/3, but also show that by incorporating regularity away from the characteristic set (in spaces of $X^{s,b}$ type) one can ``make up" the full derivative.

We next state a useful bound which is intrinsic in that frequency localization can be characterized by the functional calculus.  First note that by the normal neighborhood theorem, the inward unit normal vector field $N$ extends to a neighborhood of $\prtl  M$ such that $N$ is tangent to any geodesic intersecting the boundary orthogonally.  We observe that when applying this vector field, solutions exhibit an even stronger gain of $2^{-j/2}$ after restricting to the collar.
\begin{theorem}[Intrinsic localized energy estimates]\label{thm:intrinsic}
Suppose $\prtl M$ is strictly concave and compact, $u_\mu(t,\cdot) \in L^2(M)$ for $|t|\leq 1$, and that for some $\beta\in C_c^\infty(\mu/2,2\mu)$, $u_\mu(t,\cdot) = \beta(\mu^{-2}\Delta_g)u_\mu(t,\cdot)$. Then for some implicit constant independent of $\mu$, $j$, we have that for any  $2^{-j} \in[\mu^{-2/3},1]$
\begin{multline}\label{intrinsiclocnrg}
\mu2^{\frac j4}\|u_\mu\|_{L^2(S_{<j})}+2^{\frac j4}\|\nablatx u_\mu\|_{L^2(S_{<j})} + 2^{\frac j2}\|N u_\mu\|_{L^2(S_{<j})} \lesssim \\
\mu\|u_\mu(0,\cdot)\|_{L^2(M)}+\|\nabla_{t,x} u_\mu(0,\cdot)\|_{L^2(M)} + \|(D_t^2-\Delta_g) u_\mu\|_{L^1([-1,1];L^2( M))}
\end{multline}
where $\nabla_{t,x} u_\mu$ denotes the (intrinsically defined) space time gradient $(\prtl_t u_\mu, \nabla_g u_\mu)$.
\end{theorem}

As a consequence, we will show space time bounds for the wave equation \eqref{waveeqn} of Strichartz and square function type. The latter estimates originate in \cite{mss}, which proved these estimates for boundaryless manifolds.  When $M$ is compact, square function estimates imply $L^q(M)$ bounds on spectral clusters or quasimodes determined by $\Delta_g$.  More precisely, the eigenvalues of the nonnegative operator $\Delta_g$ can be written as an increasing sequence $\{\lambda^2_l\}_{l=1}^\infty$ with $\lambda_l \geq 0$.  Given $\lambda \geq 1$, define $\chi_\lambda$ as the operator which projects functions in $L^2( M)$ onto the subspace of eigenfunctions with the frequency of vibration satisfying $\lambda_l\in [\lambda,\lambda+1]$.  Functions in the range of $\chi_\lambda$ thus form approximate eigenfunctions in that $\|(\Delta_g-\lambda^2)\chi_\lambda\|_{L^2(M)\to L^2(M)} = O(\lambda)$.

\begin{theorem}[Intrinsic dispersive estimates]\label{thm:sqfcn}
Suppose $\prtl  M$ is strictly concave and compact.  If $u$ solves \eqref{waveeqn}, then the Strichartz estimates \eqref{wavestz} are satisfied for any scale invariant, subcritical triple $(p,q,\gamma)$ so that \eqref{wavescale} holds and $\frac 2p +\frac{n-1}{q} < \frac{n-1}{2}$.  Moreover, for any $\frac{2(n+1)}{n-1}<q\leq \infty$, we have the square function bounds
\begin{equation}\label{sqfcn}
\left\| \left(\int_{-1}^{1} |u(t,\cdot)|^2dt\right)^{\frac 12} \right\|_{L^q( M)} \lesssim \|f\|_{H^\gamma( M)} + \|g\|_{H^{\gamma-1}( M)}, \qquad \gamma = \frac{n-1}{2}-\frac{n}{q},
\end{equation}
provided the right hand side is finite.  Consequently, for any $\frac{2(n+1)}{n-1}<q\leq \infty$
\begin{equation}\label{spclus}
\|\chi_\lambda\|_{L^2( M) \to L^q( M)} \lesssim_q \lambda^{\frac{n-1}{2}-\frac{n}{q}}.
\end{equation}
\end{theorem}
The work of Smith and Sogge \cite{smithsogge07} establishes scale invariant squarefunction bounds for a smaller range of $q$, but allows for both conditions \eqref{BCs} and only assumes that $\prtl \Omega$ is $C^\infty$ and compact, that is, no concavity assumption is used.  When Dirichlet conditions are imposed and $\prtl \Omega$ is strictly concave, the estimates \eqref{sqfcn}, \eqref{spclus} are known up to the endpoint $q=\frac{2(n+1)}{n-1}$ given results of Grieser \cite{grieser} and Smith and Sogge \cite{SmSoLpreg}.  The present results thus expand the range of known exponents in these estimates when Neumann conditions are imposed.   It is now standard that \eqref{spclus} is a consequence of \eqref{sqfcn}.  For example, one can apply the bound \eqref{sqfcn} to $u(t,\cdot)=e^{it\lambda}\chi_\lambda f(\cdot)$ and use Duhamel's formula to estimate the source term.  Given constructions of spectral clusters which behave like the spherical harmonics, \eqref{spclus} is sharp for $q$ in the given range $\frac{2(n+1)}{n-1}<q\leq \infty$.  Also, interpolation with trivial $L^2$ bounds yields estimates at the Stein-Tomas endpoint
\[
\|\chi_\lambda\|_{L^2( M) \to L^{\frac{2(n+1)}{n-1}}( M)} \lesssim_\veps \lambda^{\frac{n-1}{2(n+1)}+\veps}
\]
for every $\veps>0$.  While this is not the sharp bound, it is still significant from the standpoint of Bochner-Riesz means.  Indeed, the method in \cite{soggebr} shows that this and \eqref{spclus} are enough to yield convergence of means of index $\delta > n|1/2-1/q|-1/2$ when $\max(q,q') \in [\frac{2(n+1)}{n-1},\infty]$.

\begin{remark}
In proving Theorem \ref{thm:sqfcn}, we will assume that $\gamma\in [0, 1]$.  To see that this is sufficient for the square function estimates \eqref{sqfcn}, note that if we can prove the theorem for $q$ sufficiently close to $\frac{2(n+1)}{n-1}$ then we will have $\gamma$ sufficiently close to $\frac{n-1}{2(n+1)} = \frac 12-\frac{1}{n+1}$.  Interpolating with the sharp $q=\infty$ bounds in \cite{smithsogge07} then yields the full range of estimates.
To see this for the Strichartz estimates \eqref{wavestz}, note that the subcritical condition $\frac 2p + \frac{n-1}{q} < \frac{n-1}{2}$ implies that $\gamma>\frac{n+1}{2}(\frac 12-\frac 1q)$.  When $n=2,3$, this quantity is strictly less than 1, so it suffices to establish estimates for $\gamma \in [0, 1]$ as the remaining family of estimates will follow by Sobolev embedding.  When $n\geq 4$, note that the Sobolev regularity associated to the endpoint $(p,q)=(2,\frac{2(n-1)}{n+3})$ is $\gamma = \frac 12 +\frac{1}{n-1}\in [0, 1)$. Therefore, once subcritical estimates in a deleted neighborhood of this endpoint are established, interpolation with the trivial $p=\infty$ case yields the full family of subcritical estimates.
\end{remark}

\subsection{Estimates in coordinates}\label{ss:coordintro}
Theorem \ref{thm:sqfcn} states that the scale-invariant, subcritical Strichartz bounds are satisfied in the time independent setting \eqref{waveeqn}, \eqref{wavestz}.  Our method for obtaining Strichartz bounds and localized energy estimates applies to wave equations in time dependent settings as well, but we will only state them in a non-intrinsic, coordinate dependent fashion.  Suppose $P$ is a strictly hyperbolic operator on $ \RR^{n+1}$ with principal symbol $p(x,\xi)=\sum_{i,j=0}^{n}g^{ij}(x)\xi_i\xi_j$ where $g^{ij}$ and its inverse determine quadratic forms of signature $(1,n)$.  Thus we are using $x=(x_0,x_1,\dots,x_n)$ to denote coordinates in $\RR^{n+1}$, and $x_0$ will play the role of the time variable.  Suppose $\Omega \subset \RR^{n+1}$ is a smooth domain such that the boundary $\prtl \Omega$ is time like with respect to $g^{ij}$.  We assume that the boundary is \emph{diffractive}, that is, given any local defining function $\phi$ for $\prtl \Omega$, we have $H_p^2\phi >0$.  We consider solutions to the wave equations satisfying the boundary conditions \eqref{BCs}, but note that normal vector field $N$ is defined with respect to the Lorentzian metric $g^{ij}$.

Near any point in $ \prtl \Omega$, there exists a coordinate system such that, after multiplying by a $P$ by a harmless smooth function, the operator takes the form
\begin{equation}\label{Pstz}
 P u = D_{x_n}^2u - \sum_{i,j=0}^{n-1} D_{x_i}(g^{ij}(x)D_{x_j}u) + \text{ lower order terms},
\end{equation}
(see \cite[Corollary C.5.3]{hormander3}, \cite[Lemma 4.1]{tatarubdry}). Here $x_n$ is a defining function for the boundary, that is, $\prtl \Omega$ is identified with $x_n=0$ and $\Omega$ is identified with $x_n>0$ within the coordinate system.  We may assume that that $g^{00}(x)$ is uniformly bounded from below in the coordinate system (so that the surfaces $x_0=c$ are space like) and that the quadratic form $\sum_{i,j=0}^{n-1} g^{ij}(x)\eta_i \xi_j$ has signature $(1,n-1)$. Moreover, we assume that for some $\veps$ sufficiently small,
\begin{equation}\label{c0}
\|g^{ij}-\text{diag}(1,-1,\dots,-1)\|_{Lip} \leq \veps,
\end{equation}
which can be arranged more generally by restricting to a sufficiently small cube and applying a linear transformation in $(x_0,\dots,x_{n-1})$.  Denoting $\RR^{n+1}_+=\{x:x_n>0 \}$, we may also suppose that every point in $\{x:|x|_\infty \leq 2\}\cap\RR^{n+1}_+$ identifies with a point in $\prtl \Omega$ (with $|x|_\infty = \max_{0 \leq i \leq n} |x_i|$).

\begin{theorem}[Time dependent Strichartz estimates]\label{thm:wavestz}
Suppose $u \in H^1(\overline{\RR^{n+1}_+})$ and $Pu\in L^2 (\overline{\RR^{n+1}_+})$ are compactly supported in $\{|x|_\infty \leq 1\}\cap\RR^{n+1}_+$ within the coordinate system in the preceding discussion with $u$ satisfying homogeneous Dirichlet ($u|_{x_n} = 0$) or Neumann ($\prtl_{x_n} u|_{x_n} =0$) conditions.  Then for any triple $(p,q,\gamma)$ satisfying $\frac 1p+\frac nq=\frac n2-\gamma$, the subcritical condition $\frac 2p+\frac{n-1}{q} < \frac{n-1}{2}$, and $\gamma \in [0,1]$ we have
\begin{equation}\label{coordstz}
\|u\|_{L^p([-1,1]_{x_0};W^{1-\gamma,q}(\RR^{n}_+))}
\lesssim \|u\|_{H^1(\overline{\RR^{n+1}_+})} + \|Pu\|_{L^2 (\overline{\RR^{n+1}_+})}.
\end{equation}
\end{theorem}
Here the Sobolev space $W^{s,q}(\RR^{n}_+))$, with $\RR^n_+=\{(x_1,\dots,x_n):x_n>0\}$ and $s \in [0,2]$, is defined by interpolating $L^q(\RR^{n}_+)$ and $W^{2,q}(\RR^{n}_+)$.  We take the latter to be the space of functions $f$ with weak derivatives of up to order 2 in $L^q(\RR^{n}_+)$ satisfying $f|_{x_n} = 0$ or $\prtl_{x_n} f|_{x_n} =0$.  The space $H^1(\overline{\RR^{n+1}_+})$ is defined to be the restrictions of $H^1(\RR^{n+1})$ functions to $\overline{\RR^{n+1}_+}$, and the norm can be defined by taking a sum over the $L^2(\overline{\RR^{n+1}_+})$ norm of the function and its first order weak derivatives (cf. \cite[Corollary B.2.6]{hormander3}).  Characterizing the Strichartz estimates in this fashion allows us to avoid certain technicalities regarding the Sobolev regularity of odd and even extensions of solutions.

We further note that the Strichartz estimates in Theorem \ref{thm:sqfcn} are a consequence of Theorem \ref{thm:wavestz}.  Since bounds on an open set follow from results of Kapitanski \cite{kap} and Mockenhaupt, Seeger, and Sogge  \cite{mss}, it suffices to prove estimates for solutions supported in a sufficiently small set near the boundary. We may thus take coordinates as above and assume $u$ satisfies the hypotheses of Theorem \ref{thm:wavestz} by smoothly truncating the solution in time. Since $\Delta_g$ does not depend on $t$, the coordinate transformation can be taken independent of $x_0=t$.  Assuming that $\gamma \in [0,1]$, the fact that
\[
\|u\|_{H^1(\overline{\RR^{n+1}_+})} + \|Pu\|_{L^2(\overline{\RR^{n+1}_+})} \lesssim  \|f\|_{H^1(M)} + \|g\|_{L^2(M)}                                                                                                                                                                                                                                                                                                                                                                                                                                                                                                                                                                                                                                                                                                  \]
in this setting follows from energy estimates and that $H^1(M) \subset H^1(\overline{M})$ by elliptic regularity.  Moreover, elliptic regularity and the compact support of $f,g$ also means that $W^{2,q}(\RR^{n}_+)$ can be defined equivalently as the domain of $I+\Delta_g$ in $L^q(\RR^n_+)$ subject to the corresponding boundary condition.  Consequently, we have that $\|u(x_0,\cdot)\|_{L^q(\RR^n_+)} \approx \|(I+\Delta_g)^{\frac{\gamma-1}{2}}u(x_0,\cdot)\|_{W^{1-\gamma,q}(\RR^{n}_+)}$ and the bounds \eqref{wavestz} follow from applying \eqref{coordstz} to $(I+\Delta_g)^{\frac{\gamma-1}{2}}u$.  Similar considerations hold in showing \eqref{sqfcn}, here it also suffices to assume that $u$ is compactly supported in $|x|_\infty \leq 1$ within the same coordinate system and prove the following bound over what amounts to the vector valued space $W^{1-\gamma,q}(\RR^{n}_+;L^2(-1,1)_{x_0})$:
\begin{equation}\label{sqfcncoord}
 \left(\int_{\RR^n_+} \left(\int_{-1}^1 |(I+\Delta_g)^{(1-\gamma)/2}u(x)|^2dx_0 \right)^{\frac q2}dx' \right)^{\frac 1q} \lesssim \|u\|_{H^1(\overline{\RR^{n+1}_+})} + \|Pu\|_{L^2 (\overline{\RR^{n+1}_+})}.
\end{equation}

For the purpose of proving Theorem \ref{thm:wavestz}, we may also assume that the coordinate system has been extended to all of $\overline{\RR^{n+1}_+}$ so that $g^{ij}(x) = \text{diag}(1,-1,\dots,-1)$ for $|x|_\infty \geq 3$.  This extension can of course be done in a fashion which respects \eqref{c0}.  As consistent with the above, we suppose that in spite of the extension, the diffractive hypothesis holds for $|x|_\infty \leq 2$, that is, whenever $|(\xi_0,\dots,\xi_{n-1})|=1$, then $\sum_{i,j=0}^{n-1}\prtl_{x_n}g^{ij}(x)\xi_i\xi_j$ is uniformly bounded from below by a positive constant.  We can now state our localized energy estimates in this coordinate system, where the frequency localization occurs with respect to the Fourier transform on $\RR^{n+1}$ (thus trading the compact support of $u$ for rapidly decreasing tails).  In particular, the Fourier transform in $x_n$ will be defined by taking an odd or even extension of the function in the case of Dirichlet or Neumann conditions respectively.  Theorems \ref{thm:wavestz} and \ref{thm:intrinsic} will follow as a consequence.

\begin{theorem}[Localized energy estimates in coordinates]\label{thm:maincoordestimate}
Suppose $\lambda \gg 1$ is sufficiently large and that $u: \RR^{n+1} \to \mathbb{C}$ is such that the Fourier transform is supported where $\{(\xi',\xi_n): |\xi'|\approx \lambda, |\xi_n| \lesssim \lambda\}$ and satisfies 
\begin{equation}\label{balldecay}
\left| u_\lambda (x)\right| \lesssim
\lambda^{-N}|(x_1,\dots,x_n)|_\infty^{-N}\|u_\lambda(x_0,\cdot)\|_{L^2(\RR^n_+)}, \qquad |(x_1,\dots,x_n)|_\infty \geq 3/2.
\end{equation}
Assume further that $u$ satisfies Dirichlet or Neumann conditions ($u|_{x_n} = 0$ or $\prtl_{x_n} u|_{x_n} =0$).  Then for $j$ satisfying $1 \leq 2^{-j} \leq \lambda^{-\frac 23}$, we have the bound
\begin{equation}\label{maincoordestimate}
2^{\frac j4}\|Du_\lambda\|_{L^2(S_{<j})} + 2^{\frac j2}\|D_{x_n}u_\lambda\|_{L^2(S_{<j})} \lesssim \|Du_\lambda\|_{L^2([-2,2]\times\RR^n_+)} + \|Pu_\lambda\|_{L^2([-2,2]\times\RR^n_+)},
\end{equation}
where $D=(D_{x_0},D_{x_1},\dots,D_{x_n})$ is the full space time gradient in coordinates and
\[
S_{<j} = \{x=(x_0,\dots,x_n): x_0 \in [-1,1], 0 \leq x_n \leq 2^{-j} \}.
\]
\end{theorem}

\subsection{Organization of the paper} We begin with proving the refined localized energy estimates in coordinates from Theorem \ref{thm:maincoordestimate} in \S\ref{sec:locnrg}, which is actually the heart of the matter.  In \S\ref{sec:stz}, we prove Theorems \ref{thm:wavestz} and \ref{thm:sqfcn}, working in analogy to \cite{blairrls} to obtain the subcritical Strichartz and square function bounds as a consequence of \eqref{maincoordestimate}.  The final section is then dedicated to proving Theorem \ref{thm:intrinsic} and seeing that the estimates in the first half of Theorem \ref{thm:blairrls} (namely \eqref{rlsschrod}) follow as a corollary.

\subsection*{Acknowledgement} The author thanks Hart Smith for insightful discussions concerning this work.

\section{Localized energy estimates}\label{sec:locnrg}
In this section, we prove Theorem \ref{thm:maincoordestimate} by a positive commutator argument.  In executing this approach, we expect solutions microlocalized along a light ray to satisfy better estimates when the ray reflects nontangentially on the boundary as opposed to glancing.  At the same time, an examination of the diffractive Friedlander model $D_{x_n}^2-D_{x_1}D_{x_2}-x_nD_{x_2}^2$, $x_n>0$, shows that for solutions localized on the Fourier side to  $|\xi| \lesssim \xi_2 \approx \lambda$, the uncertainty principle means that one cannot localize to scales finer than $|\xi_n|\lesssim \lambda^{2/3}$ and $x_n \lesssim \lambda^{-2/3}$ along the characteristic set.  However, the composition of pseudodifferential operators which respect this localization will not exhibit gains in the symbolic calculus relative to higher order terms in the expansion.

Given these observations, a suitable candidate for the commutant will be one of the form $Q(x,x_n,D_{x'})D_{x_n}$, the composition of a differential operator in $x_n$ and a pseudodifferential operator acting in the tangential variables for each $x_n$.  The $x_n$ dependent operator $Q$ will then have to distinguish between the glancing and nontangential behaviors.  This can be accomplished by conjugating $P$ to a normal form which resembles the diffractive Friedlander model, at which point glancing behavior can then be determined by using the equation. In particular, if we were working with the Friedlander model, we use a defining function for the characteristic set of the $x_n$-dependent tangential operator $D_{x_1}D_{x_2}^{-1}+x_n$ to determine the intersection of the characteristic sets of $D_{x_n}^2-D_{x_1}D_{x_2}-x_nD_{x_2}^2$ and $D_{x_n}^2$.  This is inspired in part by the approach to boundary trace regularity in \cite{tatarubdry}.

\subsection{Reduction to a normal form}\label{ss:symplectic}
Here we treat $\lambda$ as fixed and hence drop the subscript $\lambda$ in the notation for $u$. Given that we use space time norms $|x_0|\leq 1$ and $|x_0|\leq 2$ on the left and right respectively, we may replace $u(x)$ by $\tilde{\phi}(x_0)u(x)$, where $\tilde{\phi}$ is a bump function supported in $(-2,2)$ and identically one on $[-1,1]$ so that it now suffices to replace the $x_0$ intervals on the left and right in \eqref{maincoordestimate} by $\RR$.  We then microlocalize $u$ further so that its space-time Fourier transform is localized near a cone $|\eta_0| \approx |(\eta_1,\dots,\eta_{n-1})|$.  Indeed, outside this region, elliptic regularity gains at least half a derivative over all of $\RR^{n+1}_+$, (see for example, \cite[Theorem 2.2.B]{tay91}).  Then $u$ can further be decomposed as the sum of two functions localized to $-\eta_0 \approx |(\eta_1,\dots,\eta_{n-1})|$ and $\eta_0 \approx |(\eta_1,\dots,\eta_{n-1})|$ respectively.  Without loss of generality, we may assume $u$ is localized to the latter region.

It is convenient to change the notation in the argument which follows since the normal variable plays a special role.  We thus use $s$ to refer to the normal variable (formerly $x_n$) and $\nu$ to refer to its dual variable.  The variables $x,y\in \RR^n$ will refer to the ``tangential variables" with $x=(x_1,\dots,x_n)$ so that the ``time'' coordinate, formerly denoted by $x_0$, is now denoted $x_n$.  Using $\eta,\xi$ to denote variables dual to $y,x$, the principal symbol of the operator $P$ may be written using the summation convention as
\[
\nu^2 - g^{ij}(s,y)\eta_i\eta_j
\]
and the normal variable $s$ is now written first.  Hence up to lower order terms, $P = D_s^2-g^{ij}(s,y)D_{y_i}D_{y_j}$.  Note that in our new notation, the Fourier support of $u$ is now concentrated where $\eta_n \approx |(\eta_1,\dots,\eta_{n-1})|\approx \lambda$.

\begin{remark}\label{rem:forcetrace} Before proceeding to the heart of the matter and conjugating our problem to a normal form, we pause to comment on the regularity of the boundary trace  $\mathcal{R}(Pu) := Pu|_{s=0}$ as it will be easier to show some estimates needed in Lemma \ref{lem:bdrytrace} at this stage.  Write $Pu = F_1+F_2$ where $F_1:=\beta_\lambda^N(Pu)$ where $\beta_\lambda^N=\beta_\lambda^N(D_{s})$ is a Fourier multiplier similar to the ones defined in below \S\ref{sec:stz}, but where we take the symbol to be identically one on a neighborhood of $|\nu| \leq C\lambda$ and vanishing outside of $|\nu|\geq 2C\lambda$ for some sufficiently large constant $C$.  As mentioned in prior to Theorem \ref{thm:maincoordestimate}, the multiplier can be defined by taking a harmless even/odd extension of $Pu$ to all of $\RR^{n+1}$ and this is the one step on in the proof which uses the localization of $\widehat{u}$ to $|\nu| \lesssim \lambda$ (equivalently $|\xi_n|\lesssim \lambda$ in the old notation).  We claim that the boundary traces $F_i(s,y)|_{s=0}$ satisfy
\begin{align}
\left(\int_{s=0} |\mathcal{R}(F_1)(y)|^2\,dy \right)^{\frac 12} &\lesssim \lambda^{\frac 12}\|F_1\|_{L^2(\RR^{n+1})}\lesssim \lambda^{\frac 12}\|Pu\|_{L^2(\RR^{n+1}_+)}\label{F1bound},\\
\|F_2\|_{H^{\frac 12}(\{s=0\})} &\lesssim \|F_2\|_{H^1(\RR^{n+1})} \lesssim \lambda\|Du\|_{L^2(\RR^{n+1}_+)} \label{F2bound}.
\end{align}
To see \eqref{F1bound}, recall the formula
$
\widehat{\mathcal{R}(F_1)}(\eta)=\int_{-\infty}^\infty \widehat{F_1}(\nu,\eta)\,d\nu,
$
where $\nu$ is the dual to the normal variable. We may restrict the domain of integration to $|\nu| \lesssim \lambda$, so that Cauchy-Schwarz and Plancherel yield \eqref{F1bound} as
\[
\int|\widehat{\mathcal{R}(F_1)}(\eta)|^2 \,d\eta\lesssim \lambda \int_{-\infty}^\infty |\widehat{F_1}(\nu,\eta)|^2\,d\nu\,d\eta.
\]
Turning to \eqref{F2bound}, Sobolev trace estimates mean  that we only need to show the second inequality. We regularize the coefficients of $P$ to frequencies less than $c\lambda$, for some sufficiently small $c$, and denote $P_\lambda$ as the result of replacing the coefficients of $P$ by their smooth counterparts.  This yields
\[
F_2=(1-\beta_\lambda^N)Pu = (1-\beta_\lambda^N)(P-P_\lambda)u
\]
as the Fourier transform of $P_\lambda u$ is supported where $|\nu|\leq 16\lambda$. Now let $g^{ij}_\lambda$ denote the regularization of $g^{ij}$, which satisfies $\|g^{ij}-g^{ij}_\lambda\|_{L^\infty}\lesssim \lambda^{-1}$ and $\|Dg^{ij}-Dg^{ij}_\lambda\|_{L^\infty}\lesssim 1$.  Since derivatives falling on $u$ yield a loss of $\lambda$, \eqref{F2bound} follows from
\begin{equation*}
\|(g^{ij}-g^{ij}_\lambda)D_{x_i}D_{x_j} Du\|_{L^2}+ \|(Dg^{ij}-Dg^{ij}_\lambda)D_{x_i}D_{x_j} u\|_{L^2}
\lesssim \lambda\|Du\|_{L^2}.
\end{equation*}
\end{remark}

As discussed above, a key idea in the proof of Theorem \ref{thm:maincoordestimate} is to apply a unitary Fourier integral operator which takes $P$ to a normal form resembling the diffractive Friedlander model.  This is inspired by the strategy in \cite{tatarubdry}, but in contrast to that work, we take a Fourier integral operator independent of $s$ (and $D_s$), which conjugates $g^{ij}(0,y)D_{x_i}D_{x_j}$ to a normal form rather than taking a family of such operators which depend on $s$.  For the immediate discussion, we denote $a$ as the quadratic form in $\eta$, $a(y,\eta)=g^{ij}(0,y)\eta_i\eta_j$.  Recall that the Fourier support of $u$ is supported in the conic region $\eta_n \approx |(\eta_1,\dots,\eta_{n-1})|$. Prescribing the data for the equation $\{a,b\}=0$ suitably, there exists a solution $b(y,\eta)>0$ which is homogeneous of degree one such that $b(y,\eta)\approx |\eta|$ .  As in \cite[Lemma 4.2]{tatarubdry}, we now set $\xi_1 = ab^{-1}$, $\xi_2 = b$, which extends to a homogeneous canonical transformation $(y,\eta)\mapsto (x,\xi)$ defined on the conic region $\eta_n \approx |(\eta_1,\dots,\eta_{n-1})|$ by \cite[Theorem 21.1.9]{hormander3}).  Using the summation convention in $i,j$, we have
\[
g^{ij}(s,y)\eta_i\eta_j = g^{ij}(0,y)\eta_i\eta_j + s \int_0^1 \prtl_s g^{ij}(st,y)\eta_i\eta_j\,dt = \xi_1 \xi_2 + s\tilde{r}(s,x,\xi).
\]
By the diffractive hypothesis $\tilde{r}(s,\cdot)\in S^2_{1,0}(\RR^{2n}_{x,\xi})$, defines an elliptic symbol satisfying for some $c_0>0$
\begin{equation}\label{c0lower}
r(s,x,\xi):=\xi_2^{-2}\tilde{r}(s,x,\xi)\geq c_0.
\end{equation}

Consequently, there exists a unitary Fourier integral operator $T:L^2(\RR^n_x) \to L^2(\RR^n_y)$ such that conjugating $P$ by this operator for each $s$ yields the pseudodifferential operator
\begin{equation}\label{conjugoperator}
D_s^2 - D_{x_1}D_{x_2}-s\tilde{R}(s,x,D)
\end{equation}
where $\tilde{R}$ is an elliptic operator with principal symbol $\tilde{r}(s,\cdot) \in S^0_{1,0}(\RR^{2n}_{x,\xi})$. Strictly speaking, we should add on an operator $R_1(s,x,D)$ to \eqref{conjugoperator} such that for each $s$, $R_1(s,\cdot)\in \text{Op}(S^1_{1,0})$, but the error here is harmless as it can be absorbed into $Pu$.

Recall that in the original coordinates, $\widehat{u}$ is supported in the region $\eta_n \approx |(\eta_1,\dots,\eta_{n-1})|\approx \lambda$.  Moreover, regarding as a $\xi_2$ a positive function of $y,\eta$, we have $\xi_2 \approx |\eta|$.  Therefore there exists a smooth bump function $\beta_\lambda(\xi)$ such that
\begin{equation}\label{betadef}
\supp(\beta_\lambda) \subset \{\xi\in \RR^n: |\xi_1| \lesssim \xi_2 \approx |\xi| \} \cap \{\xi\in \RR^n: |\xi| \approx \lambda \}
\end{equation}
and for each $s$,
\begin{equation}\label{betaerror}
\|(I-\beta_\lambda(D))T^{-1}u(s,\cdot)\|_{L^2(\RR^n)} \lesssim_N \lambda^{-N} \|u(s,\cdot)\|_{L^2(\RR^n)}.
\end{equation}
In the remainder of this section, we may now replace $u(s,\cdot)$ by $\beta_\lambda(D)T^{-1}u(s,\cdot)$ as this $T$ clearly preserves $L^2$ norms. As noted above, we may assume that $P$ is exactly the pseudodifferential operator in \eqref{conjugoperator}.

Set $J_\lambda = \lfloor \log_2 \lambda^{2/3} \rfloor$ and define the following variations on the $S_{<j}$ from above
\[
\tilde{S}_{j} = \{(s,x):\,2^{-j-1}\leq s \leq 2^{-j} \}, \qquad \tilde{S}_{<J_\lambda}= \{(s,x):\,0\leq s \leq 2^{-J_\lambda} \},
\]
the former being defined for for $1 \leq j < J_\lambda$.  By geometric summation, \eqref{maincoordestimate} is now a consequence of showing that for some implicit constant independent of $k$, we have
\begin{equation*}
 2^{\frac k2}\|D_{x_2}u\|_{L^2(\tilde{S}_k)}^2 + 2^{k}\|D_{s}u\|_{L^2(\tilde{S}_k)}^2 \lesssim \|Du\|_{L^2}^2 + \|Pu\|_{L^2}^2,
\end{equation*}
and that the same holds over $\tilde{S}_{<J_\lambda}$.  In either case, we assume that the $L^2$ norms on the right are taken over $\RR^{n+1}_+$. We now introduce the method of slowly varying sequences of Tataru \cite[\S4]{tataruschrod}.  It suffices to show that
\begin{multline}\label{alphajbound}
\sum_{j=1}^{J_\lambda} a_j \left(\|(\lambda^{-\frac 23} + s)^{-1/4}D_{x_2}u\|_{L^2(\tilde{S}_j)}^2+ \|(\lambda^{-\frac 23} + s)^{-1/2}D_{s}u\|_{L^2(\tilde{S}_j)}^2  \right)\\
\lesssim \|Du\|_{L^2}^2 + \|Pu\|_{L^2}^2
\end{multline}
where $\{a_j\}_{j=1}^{J_\lambda} $ is defined by $a_j= \delta_{jk}$ and the $j=J_\lambda$ term in the sum is understood to involve the norm over $L^2(\tilde{S}_{<J_\lambda})$. Let $\delta>0$ be sufficiently small but fixed.  A positive sequence is said to be \emph{slowly varying} if
\begin{equation}\label{slowvary}
\left| \log_2 \alpha_j - \log_2 \alpha_{j-1}\right|\leq \delta , \qquad \text{equivalently,}\qquad
2^{-\delta}\leq \frac{\alpha_j}{\alpha_{j-1}} \leq 2^\delta.
\end{equation}
We claim that any sequence of the form $a_j = \delta_{jk}$ above can be dominated by a slowly varying sequence satisfying $\alpha_{J_\lambda} =1$.  Indeed, for $1 \leq j \leq J_\lambda$, we may define
\[
\alpha_j = 2^{-|j-k|\delta} + (1-2^{-(J_\lambda-k)\delta})2^{-(J_\lambda-j)\delta},
\]
which satisfies the desired requirements.  It now suffices to prove \eqref{alphajbound} with $a_j$ replaced by $\alpha_j$.  The $\alpha_j$ will also satisfy $\sum_{j=1}^{J_\lambda}\alpha_j \lesssim 1$, with implicit constant independent of $k$.

We now associate a smooth function $\alpha(s)$ for $s>0$  to the sequence $\{ \alpha_j\}$ with the properties
\[
\alpha(s) \approx \alpha_j, \quad\text{for } s\approx 2^{-j}\in[\lambda^{-\frac 23},1],
\]
\[
\alpha(s) \equiv 1 \quad \text{for } s \in [0,\lambda^{-\frac 23}],
\]
\[
|\alpha^{(k)}(s)| \lesssim (\lambda^{-\frac 23}+s)^{-k}\alpha(s).
\]
Moreover, we may assume that \eqref{alphafcnvary} below is satisfied.  Such a function can be constructed by beginning with a step function which satisfies the first two requirements, then convolving with a bump function compactly supported in $|s| \ll 2^{-j}$ near the discontinuity at $s=2^{-j}$.  Recalling that $\sum_{j=1}^{J_\lambda}\alpha_j \lesssim 1$, we thus have that
for any $s \in [0,1]$ and any large constant $C$,
\[
\int_{0}^s \frac{\alpha(t)}{C\lambda^{-\frac 23} + t}\,dt \lesssim 1.
\]
This highlights the tradeoff in working with slowly varying sequences: multiplying by $\alpha(t)$ corrects the nonintegrability of $t \mapsto 1/t$ at the cost of obtaining an estimate $\ell^\infty_j$ rather than $\ell^p_j$ for some $p$.

Given that we may assume $\widehat{u}(s,\cdot)$ is localized in a cone (cf. \eqref{betaerror}),
\begin{equation}\label{freqloc}
\xi_2 \gg |(\xi_1, \xi_3, \dots, \xi_n)| \qquad \text{and} \qquad \xi_2 \approx \lambda
\end{equation}
the inequality \eqref{alphajbound} is now further reduced to
\begin{equation}\label{alphahalfspace}
\int_{\RR^{n+1}_+}\alpha(s)\left(\frac{|D_{x_2} u(s,x)|^2}{(C\lambda^{-\frac 23} + s)^{\frac 12}} +\frac{|D_s u(s,x)|^2}{C\lambda^{-\frac 23} + s} \right)\,dsdx \lesssim \|Du\|_{L^2}^2 + \|Pu\|_{L^2}^2.
\end{equation}
In the next section, we will use a positive commutator method to show a variation on this estimate which will suffice.

\subsection{The Positive Commutator Method}\label{ss:poscomm}
We are now in a position to define the commutant which will yield the desired energy estimates after commuting with $P$.  In this subsection, we assume that pseudodifferential operators are defined using the Weyl-quantization, so that operators with real symbols are self-adjoint, which we use without further reference.

Let $C_1$ and $C_2$ be large constants with $C_1 \gg C_2$ and take $\psi$ to be the function
\begin{equation}\label{psidefn}
\psi(s) := C_2+ \int_0^{s} \frac{\alpha(t)}{C_1\lambda^{-\frac 23}+t}\,dt.
\end{equation}
Recalling the definition of $c_0$ from \eqref{c0lower}, let $\tilde{\zeta}$ be an increasing, smooth cutoff on $\RR$ supported in $[-\frac {c_0}2,\infty)$ and identically one on $[-\frac {c_0}4,\infty)$. Also, let $\chi \in C^\infty(\RR)$ be supported in $[1,\infty)$ and identically one on $[2,\infty)$.
Next, let $\zeta$ be a smooth cutoff to the region
\[
\zeta(s,\xi_1,\xi_2) := \tilde{\zeta}\left(\frac{\xi_1}{\xi_2}s^{-1}\right)\chi(\lambda^{\frac 23}s)+\chi(\lambda^{-\frac 13}\xi_1)(1-\chi)(\lambda^{\frac 23}s).
\]
One feature of $\zeta$ is that for every $s \geq 2\lambda^{-\frac 23}$, it satisfies
\[
\frac{\xi_1}{\xi_2r(s,x,\xi)}\zeta+s \geq \frac 12 s, \qquad \text{equivalently}, \qquad \xi_1\xi_2\zeta+s\tilde{r}(s,x,\xi)\geq \frac 12 s\tilde{r}(s,x,\xi)
\]
where $r = \xi_2^{-2}\tilde{r}$ is defined in \eqref{c0lower}.  This shows that along the characteristic set of the operator in \eqref{conjugoperator}, $\zeta$ truncates away from $\nu/\xi_2=0$ on the scale of $s$, amounting to the ``hyperbolic" region when $0 \leq s \leq 2\lambda^{-2/3}$.  We further let
\begin{equation}\label{qtildedef}
\tilde{q}(s,x,\xi) :=\left(C_1\lambda^{-\frac 23} + \frac{\xi_1}{\xi_2r(s,x,\xi)}\zeta+s \right)^{-\frac 12}\beta_\lambda(\xi),
\end{equation}
where $\beta_\lambda(\xi) \in S^0_{1,0}(\RR^{2n}_{x,\xi})$ is the smooth cutoff defined in \eqref{betadef}.  Given these definitions, we let $q(s,x,\xi):= \tilde{q}(s,x,\xi) \psi(s)$ let $Q$ be the operator acting on functions $w(s,y)$ on $\RR^{n+1}_+$ defined by
\begin{equation}\label{fourierintegral}
Qw(s,x) = \frac{1}{(2\pi)^n}\iint e^{i(x-y)\cdot\xi} q\Big(s,\frac{x+y}{2},\xi\Big)w(s,y)\,dyd\xi,
\end{equation}
so that $Q$ acts as an $s$-dependent pseudodifferential operator in the $x$ variables.  Also let $\tilde{Q}$ be the operator defined by the same integral but with $q$ replaced by $\tilde{q}$, so that $Q$ is the result of multiplying $\tilde{Q}$ by the function $\psi(s)$.

We begin with a discussion of the regularity of the symbols here.  It is verified that $\zeta(s,\cdot)\in S_{1/3,0}^{0}$ as each differentiation in $\xi_1$ gains a power of $(\lambda^{\frac 13}+\xi_2s)^{-1}$ while each differentiation in the other directions results in a stronger gain of a power of $\xi_2^{-1}$.  More generally, we have that for $k\geq 1$, $\prtl_{s}^k\zeta(\RR^{2n}_{x,\xi})$, are supported in $s \geq \lambda^{-\frac 23}$ and $|\xi_1\xi_2^{-1}|^k \prtl_{s}^k\zeta \in S_{1/3,0}^{0}(\RR^{2n}_{x,\xi})$.

We next observe that for each $s$, $q(s,\cdot)$ is a symbol in $S^{1/3}_{1/3,0}(\RR^{2n}_{x,\xi})$ with respect to the tangential variables.  Indeed, rewriting $q$ as
\[
q(s,x,\xi)=\psi(s)\xi_2^{\frac 12}\left(C_1\lambda^{-\frac 23}\xi_2 + \xi_1r^{-1}\zeta+\xi_2 s \right)^{-\frac 12}\beta_\lambda(\xi),
\]
it can be seen that differentiating $q$ in $\xi_1$ gains a power of
\[
\left(C_1\lambda^{-\frac 23}\xi_2 + \xi_1r^{-1}\zeta+\xi_2 s \right)^{-1} \lesssim (\lambda^{\frac 13} + \xi_2s)^{-1}
\]
while in the other directions we have a stronger gain of $|\xi|^{-1}\approx \xi_2^{-1}\approx \lambda^{-1}$. The symbols could be described more precisely using the weight vectors introduced by Beals \cite{bealscalculus}, or the $S(m,g)$ classes of H\"ormander \cite{hormanderweyl}, though this is not needed due to the frequency localization of the problem.  Moreover, the regularity of the symbol improves as $s$ increases.

More generally, it can be verified that for each $s$,
\begin{equation}\label{qregularity}
\psi^{(k_1)}(s)\prtl_s^{k_2} \tilde{q}(s,\cdot) \in S_{\frac 13,0}^{\frac 13(k_1+k_2+1)}(\RR^{2n}_{x,\xi}),
\end{equation}
with the same gains when differentiating in $\xi_1$ and $\xi_2, \dots, \xi_n$ so that $\prtl_s q(s,\cdot)$ and $\prtl_s^2 q(s,\cdot)$ define symbols in $S^1_{\frac 13,0}$ and $S^{\frac 53}_{\frac 13,0}$ respectively.  However, the regularity improves when multiplying by a power of $s$, for example, for each $s$
\begin{equation}\label{qregimproved}
s^{\frac 12}q(s,\cdot)\in S_{\frac 13,0}^{0}(\RR^{2n}_{x,\xi}), \qquad s \prtl_s q(s,\cdot)\in S_{\frac 13,0}^{\frac 13}(\RR^{2n}_{x,\xi}), 
\end{equation}

We consider a pseudodifferential operator of the form
\[
(QD_s)^w u  = \frac 12 Q D_s  u + \frac 12 D_s (Q u) = QD_s u + \frac 12 (D_s Q) u ,
\]
where the superscript $w$ in $(QD_s)^w$ is used to emphasize that we take the Weyl quantization of $q(s,x,\xi)\cdot\nu $, resulting in a differential operator in the $s$ variable.  We also remark that $(D_s Q) $ denotes the pseudodifferential operator with symbol $D_s q$ (equivalently $[D_s,Q]$), and a similar convention will be taken below (e.g. the operator $\prtl_s^2Q$ in \eqref{secondsidebdry} below is the pseudodifferential operator with symbol $\prtl_s^2q$). The positive commutator strategy is thus to prove suitable upper and lower bounds on
\begin{equation}\label{commutator}
\Re \left\langle i[P,(QD_s)^w] \,u, u \right\rangle,
\end{equation}
where $\langle \cdot,\cdot \rangle$ denotes the standard inner product on $L^2(\RR^{n+1}_+)$.

\begin{remark}
The choice of symbol $q$ determined by \eqref{qtildedef} can be motivated by considering the diffractive Friedlander model, the differential operator with symbol $p = \nu^2-\xi_1\xi_2-s\xi_2^2$, though the discussion here will restrict attention to the first term in the integrand on the left in \eqref{alphahalfspace}.  If one takes the ansatz that the commutant should take the form $q(s,\xi_1,\xi_2)\nu$ (which would be suitable for this model), then formally the principal symbol on the of the commutator in \eqref{commutator} is
\[
H_p\left( q(s,\xi)\nu \right) = 2\nu^2\prtl_s q(s,\xi) + \xi_2^2 q(s,\xi)
=\xi_2^2\left[2\left(\frac{\xi_1}{\xi_2}+s\right)\prtl_s q(s,\xi) +  q(s,\xi)\right].
\]
where the second identity follows by using that $\nu^2=\xi_1\xi_2+s\xi_2^2$ along the characteristic set. Here we factor out $\xi_2^2$ as it yields the derivative $D_{x_2}$ which will fall on $u$.  The method of integrating factors gives that the solution to $2(\frac{\xi_1}{\xi_2}+s)\prtl_s q +  q = f(s,\xi_1,\xi_2)$ when $\frac{\xi_1}{\xi_2}+s>0$ is
\begin{equation}\label{qansatz}
q(s,\xi_1,\xi_2) = \frac 12\left(\frac{\xi_1}{\xi_2}+s\right)^{-1/2} \int_0^s  \left(\frac{\xi_1}{\xi_2}+t\right)^{-1/2} f(t,\xi_1,\xi_2)\,dt ,
\end{equation}
up to terms independent of $s$. Recalling \eqref{alphahalfspace}, we want to choose $f$ such that $f \gtrsim \alpha(s)(C_1\lambda^{-2/3}+s)^{-1/2}$.  Taking equality here satisfies this objective, but the approach taken so far only works away from the zero set of $\xi_1\xi_2^{-1}+s$, or equivalently $\nu=0$.  Selecting \[
f(s,\xi_1,\xi_2)=\frac{(\frac{\xi_1}{\xi_2}+s)^{\frac 12}\alpha(s)}{C_1\lambda^{-\frac 23}+s}
\]
means that the integral in \eqref{qansatz} is independent of $\xi_1,\xi_2$ and is effective provided $\xi_1/\xi_2+s \gtrsim C_1\lambda^{-2/3}+s$. At the same time, since $|\nu| = (\xi_1/\xi_2+s)^{1/2}|\xi_2|$ along the characteristic set, we should have suitable upper bounds on the commutator since this formally means $QD_s \lesssim D_{x_2}$.  At the other extreme, when $\xi_1/\xi_2+s$ is small in that $\xi_1/\xi_2+s\leq 2 (C_1\lambda^{-2/3}+s)$, taking $q =\psi(s)(C_1\lambda^{-2/3}+s)^{-\frac 12}$ independent of $\xi_1$, $\xi_2$ with $\psi$ as in \eqref{psidefn}, it can be verified that $H_p(q\nu) \gg \alpha(s)(C_1\lambda^{-2/3}+s)^{-1/2}$.  The choice of $q$ determined by \eqref{qtildedef} is thus the result smoothly transitioning between these two extremes by employing the $\zeta$ cutoff while respecting the uncertainty principle.  Below it will be seen that introducing the $\zeta$ into the symbol presents acceptable error.
\end{remark}

\subsubsection{Upper bounds on the commutator} We first prove upper bounds on \eqref{commutator}, which does not use the structure of the commutator, but instead we integrate by parts to dominate the expression by $\|Du\|_{L^2}^2 + \|Pu\|_{L^2}^2$.  Begin by noting the following integration by parts formulae (with $\int_{s=0}f\,dx$ denoting the integral of the function $x\mapsto f(s,x)|_{s=0}$ over $\RR^n$):
\[
\langle D_s v, w \rangle = \langle v, D_s w \rangle +i \int_{s=0}v\overline{w}\,dx
\]
\begin{equation*}
\langle (QD_s)^w v, w \rangle = \langle v, (QD_s)^w w \rangle + i \int_{s=0} v \overline{Qw}\,dx = \langle v, (QD_s)^w w \rangle + i \int_{s=0} Qv \overline{w}\,dx
\end{equation*}
\[
\langle D_s^2 v, w \rangle = \langle v, D_s^2 w \rangle + \int_{s=0}\prtl_s v\overline{w}\,dx-\int_{s=0}v\overline{\prtl_s w}\,dx.
\]
Thus since $P-D_s^2$ is an operator acting in the tangential variables, we use the boundary condition as appropriate to obtain
\begin{align*}
\langle P(QD_s)^w u, u \rangle &= \langle (QD_s)^w u, Pu \rangle + \int_{s=0}\prtl_s (QD_s)^wu\overline{u}\,dx - \int_{s=0}(QD_s)^wu\overline{\prtl_s u}\,dx,\\
\langle (QD_s)^wP u, u \rangle
&= \langle P u, (QD_s)^wu \rangle + i \int_{s=0}Pu\overline{Qu}\,dx.
\end{align*}
Combining the two thus yields
\begin{multline}\label{energysetup}
\left\langle i[P,(QD_s)^w] \,u, u \right\rangle =   2\Im \langle Pu, (QD_s)^w u\rangle \\ -i\int_{s=0}(QD_s)^wu\overline{\prtl_s u}\,dx + i\int_{s=0}\prtl_s (QD_s)^wu\overline{u}\,dx + \int_{s=0}Pu\overline{Qu}\,dx.
\end{multline}

To control the boundary integrals, we need the following theorem of D. Tataru:
\begin{theorem}[\cite{tatarubdry}]\label{thm:tatarubdry}
Suppose $u \in H^1(\overline{\RR^{n+1}_+})$ and $Pu \in L^2(\RR^{n+1}_+)$.
\begin{enumerate}
 \item Suppose $u$ satisfies Dirichlet conditions, then the normal derivative satisfies
\begin{equation}\label{dirichletnormal}
  \|\prtl_su\|_{L^2(\{s=0\})} \lesssim \|u\|_{H^1} + \|Pu\|_{L^2}
\end{equation}
\item Suppose $u$ satisfies Neumann conditions.  Then
\begin{equation}\label{neumanntrace}
  \|u\|_{H^{\frac 56}(\{s=0\})} =\|\langle D_T \rangle^{\frac 56} u\|_{L^2(\{s=0\})} \lesssim \|u\|_{H^1} + \|Pu\|_{L^2}
\end{equation}
(where $\langle D_T \rangle$ denotes the tangential Sobolev weight/Fourier multiplier with symbol $(1+|\xi|^2)^{\frac 12}$).
Moreover, let $B$ denote the pseudodifferential operator with symbol  $b(\xi)=\langle \xi \rangle^{\frac 56} \langle \xi_2^{-\frac 13}\xi_1 \rangle^{\frac 14}$ and consider the $H^{1,\frac 56}$ spaces defined in \cite[p.195]{tatarubdry}, stating that $u\big|_{s=0} \in H^{1,\frac 56}$ if and only if $Bu \in L^2$.  Then (see p. 203 of that same work)
\begin{equation}\label{neumannnullform}
\|Bu\|_{L^2(\{s=0\})} = \|u\|_{H^{1,\frac 56}(\{s=0\})} \lesssim \|u\|_{H^1} +  \|Pu\|_{L^2}.
\end{equation}
\end{enumerate}

\end{theorem}

For \eqref{neumannnullform}, we recall that the unitary transformation $T$ defined in \S\ref{ss:symplectic} agrees with the one in \cite{tatarubdry} when $s=0$.  We now use these results to obtain estimates on the boundary integrals needed in our argument.

\begin{lemma}\label{lem:bdrytrace}
There exists a uniform constant $C$ such that
\begin{equation}\label{dirichlettrace}
-\Re\left(\int_{s=0}i(QD_s)^wu\overline{\prtl_s u}\right) \leq  C\left(\|Du\|_{L^2}^2 + \|Pu\|_{L^2}^2\right) .
\end{equation}
Moreoever, the boundary traces of $u$ and its derivatives satisfy the following bounds:
\begin{align}
\left|\int_{s=0}u\overline{(\prtl_s^2Q)u}\right|&\lesssim \|Du\|_{L^2}^2 + \|Pu\|_{L^2}^2, \label{secondsidebdry}\\
\left|\int_{s=0}Pu\overline{Qu}\right|&\lesssim \|Du\|_{L^2}^2 + \|Pu\|_{L^2}^2, \label{forcetrace}\\
\left|\int_{s=0}i\prtl_s (QD_s)^wu\overline{u}\right| &\lesssim \|Du\|_{L^2}^2 + \|Pu\|_{L^2}^2, \label{maintrace}
\end{align}
\end{lemma}

\begin{proof}
We first observe that \eqref{dirichlettrace} is only nontrivial if Dirichlet conditions are imposed.  In this case, since $q(0,\cdot) \in S_{1/3,0}^{1/3}$ is nonnegative, the sharp G\r{a}rding inequality (see e.g. \cite[Proposition 0.7B]{tay91}) shows that
\[
-\Re\left(\int_{s=0}i(QD_s)^wu\overline{\prtl_s u}\right)=-\Re\left(\int_{s=0}(Q\prtl_s u)\overline{\prtl_s u}\right) \lesssim \|\prtl_su\|_{L^2(\{s=0\})}^2
\]
The estimate \eqref{dirichlettrace} now follows from \eqref{dirichletnormal}.

The remaining bounds are only nontrivial if Neumann conditions are imposed, so this will be assumed for the remainder of the proof.  The inequality \eqref{secondsidebdry} is a consequence of \eqref{neumanntrace} and the fact that $\prtl^2_s q(0,\cdot) \in S^{5/3}_{1/3,0}(\RR^{2n}_{x,\xi})$:
\[
\left|\int_{s=0}u\overline{\prtl_s^2Qu}\right| \leq \|\langle D_T \rangle^{-\frac 56} (\prtl_s^2 Q)u\|_{L^2(\{s =0\})} \|\langle D_T \rangle^{\frac 56}  u\|_{L^2(\{s =0\})} \lesssim \|u\|_{H^{\frac 56}(\{s=0\})}^2.
\]

Turning to \eqref{forcetrace}, we recall the decomposition $Pu = F_1+F_2$ and bounds \eqref{F1bound} and \eqref{F2bound} from Remark \ref{rem:forcetrace}.  Both of these bounds are preserved by the action of the unitary operator defined in \S\ref{ss:symplectic}.  The frequency localization again gives that $\lambda^{1/2}Q|_{s=0} \in \text{Op}(S^{5/6}_{1/3,0})$ so that by Cauchy-Schwarz and \eqref{neumanntrace}, we have
\[
\left|\int_{s=0}F_1\overline{Qu}\right|\lesssim \|Pu\|_{L^2(\RR^{n+1}_+)}\| u\|_{H^{\frac 56}(\{s=0\})}\lesssim \|Du\|_{L^2}^2 + \|Pu\|_{L^2}^2.
\]
To handle $F_2$, we use \eqref{F2bound} and that $\lambda\langle D_T\rangle^{-\frac 12} Q|_{s=0} \in \text{Op}(S^{\frac 56}_{\frac 13,0})$:
\begin{equation*}
\left|\int_{s=0}F_2\overline{Qu}\right| \lesssim \|Du\|_{L^2(\RR^{n+1}_+)}\| u\|_{H^{\frac 56}(\{s=0\})} \lesssim \|Du\|_{L^2}^2 + \|Pu\|_{L^2}^2.
\end{equation*}

We finally turn to the bound in \eqref{maintrace}.  Observe that
\begin{equation}\label{maintraceexpn}
i\prtl_s (QD_s)^wu = \frac 12\prtl_s(Q\prtl_s u)+\frac 12 \prtl_s^2 (Qu)=Q\prtl_s^2u + \frac 32 (\prtl_sQ)\prtl_su + \frac 12 (\prtl_s^2Q)u.
\end{equation}
Given the boundary condition and \eqref{secondsidebdry}, we only need to treat the first term here which, by the equation $P\big|_{s=0} = D_s^2 -D_{x_1}D_{x_2}$, reduces to estimating
\[
\left|\int_{s=0} Q(D_{x_1}D_{x_2}u)\overline{u} \right|+\left|\int_{s=0} QPu\overline{u} \right|,
\]
and the second term is handled using \eqref{forcetrace}. The principal symbol of $QD_{x_1}D_{x_2}$ at $s=0$ is
\[
C_2\left(\lambda^{-\frac 23} +  \frac{\xi_1\chi(\lambda^{-\frac 13}\xi_1)}{\xi_2r(0,x,\xi)} \right)^{-\frac 12}\xi_1\xi_2\beta_\lambda(\xi).
\]
We now use the operator defined in \eqref{neumannnullform}, observing that the principal symbol of the triple composition $B^{-1}(QD_{x_1}D_{x_2}) B^{-1}$ is
\[
\langle \xi \rangle^{-\frac 53} \langle \xi_2^{-\frac 13}\xi_1 \rangle^{-\frac 12}\left(\lambda^{-\frac 23} +  \frac{\xi_1\chi(\lambda^{-\frac 13}\xi_1)}{\xi_2r(0,x,\xi)}\right)^{-\frac 12}\xi_1\xi_2\beta_\lambda(\xi) \in S^0_{\frac 13,0}(\RR^{2n}_{x,\xi}).
\]
Consequently,
\[
\left|\int_{s=0} Q(D_{x_1}D_{x_2}u)\overline{u} \right|=\left|\int_{s=0} B^{-1}(QD_{x_1}D_{x_2})B^{-1} (Bu)\overline{Bu} \right|\lesssim \|Bu\|_{L^2(\{s=0\})}^2,
\]
which by \eqref{neumannnullform}, concludes the proof of \eqref{maintrace}.
\end{proof}

Recalling that $\prtl_s q(s,\cdot) \in S_{1/3,0}^1(\RR^{2n}_{x,\xi})$, and the definition of $\tilde{Q}$ following \eqref{fourierintegral}, we have that
\begin{equation}\label{qweylbound}
\|(QD_s)^wu\|_{L^2(\RR^{n+1}_+)}\lesssim \|\tilde{Q}D_s u\|_{L^2(\RR^{n+1}_+)} + \|D u\|_{L^2(\RR^{n+1}_+)}.
\end{equation}
Consequently, given the boundary trace estimates, we now have that \eqref{energysetup} yields
\begin{equation}\label{realupperbound}
\Re \left\langle i[P,(QD_s)^w] \,u, u \right\rangle \lesssim \|\tilde{Q}D_s u\|_{L^2}\|Pu\|_{L^2} + \|Du\|_{L^2}^2+\|Pu\|_{L^2}^2
\end{equation}
with all $L^2$ norms on the right taken over $\RR^{n+1}_+$.  The first term on the right here will be treated in \eqref{normalupperbound} below.

\subsubsection{Lower bounds on the commutator} We now view $i[P,(QD_s)^w] $ as a positive operator, and compute the left hand side of \eqref{realupperbound} in a manner which reflects this.    First compute
\begin{multline*}
-i[s\tilde{R}+D_{x_1}D_{x_2},(QD_s)^w ]u=[Q\prtl_s+\frac 12(\prtl_s Q) ,s\tilde{R}+D_{x_1}D_{x_2} ]u=\\
Q\tilde{R}u +sQ (\prtl_s\tilde{R})u + [Q,D_{x_1} D_{x_2}+s\tilde{R}]\prtl_s u+\frac 12  [(\prtl_s Q),D_{x_1}D_{x_2}+s\tilde{R}]u.
\end{multline*}
Next we compute
\begin{align*}
[D_s^2, (QD_s)^w]u &= \frac 12\left(D_s^2(QD_s u)+ D_s^3(Qu) - QD_s^3 u - D_s(QD_s^2u)\right)\\
& = \frac 12\left([D_s^3,Q]u+[D_s^2,Q]D_s u-[D_s,Q]D_s^2u\right)\\
& = 2(D_sQ)D^2_su + 2(D_s^2Q)D_su +\frac 12(D_s^3Q)u,
\end{align*}
where the last equality is a consequence of expanding each commutator in the second line and collecting like terms. Using that $Q$ is the product of $\psi(s)$ and the operator $\tilde{Q}$, the last row here can be rewritten as
\[
2(D_s\psi)(s)\tilde{Q}D^2_su + 2\psi(s)(D_s\tilde{Q})D^2_su+ 2(D_s^2Q)D_su +\frac 12(D_s^3Q)u.
\]
For the second term here, we replace $D^2_su=(D_{x_1}D_{x_2} + s \tilde{R})u + Pu$ using the equation.  For the first term here we split its contribution in half, making the same replacement for half the terms, but not the other.  The result is
\begin{multline*}
\langle i[D_s^2,(QD_s)^w]u, u \rangle = \langle\psi'(s)\tilde{Q}D^2_su,u\rangle \\ + \left\langle\left(\psi'(s)\tilde{Q}+ 2\psi(s)(\prtl_s\tilde{Q})\right)(D_{x_1}D_{x_2} + s \tilde{R})u,u \right\rangle\\  - 2\langle (\prtl_s^2Q)\prtl_su,u\rangle -\frac 12\langle(\prtl_s^3Q)u,u\rangle + \left\langle\left(\psi'(s)\tilde{Q}+ 2\psi(s)(\prtl_s\tilde{Q})\right)Pu,u \right\rangle
\end{multline*}
However, we note that
\begin{align*}
\Re\langle (\prtl_s^2Q)\prtl_su,u\rangle
&= \langle (\prtl_s^2Q)\prtl_su,u\rangle - \langle \prtl_s(\prtl_s^2Qu),u\rangle + \int_{s=0}u\overline{\prtl_s^2Qu}\,dx\\
&= -\langle (\prtl_s^3Q) u,u\rangle + \int_{s=0}u\overline{\prtl_s^2Qu}\,dx
\end{align*}
which eliminates the $\prtl_s^2Q \prtl_su$ term at the cost of adding a boundary term. Moreover, after an integration by parts which generates trivial boundary terms, we obtain
\[
\langle\psi'(s)\tilde{Q}D^2_su,u\rangle = \langle\psi'(s)\tilde{Q}D_su,D_s u\rangle +
\langle (\psi''(s)\tilde{Q}+\psi'(s)\prtl_s\tilde{Q})\prtl_s u , u \rangle
\]
and the real part of the second term on the right hand can be reduced to a sum of terms of the form $\langle(\psi^{(k_1)}(s)\prtl_s^{k_2}\tilde{Q})u,u\rangle$, $k_1+k_2 =3$, the same argument used to eliminate $\Re\langle (\prtl_s^2Q)\prtl_su,u\rangle$.

In summary, using harmless coefficients $c_{2,k_1,k_2}$, $c_{3,k_1,k_2}$, we may write \eqref{commutator} as
\begin{multline}\label{positivesetup}
\Re\langle\psi'(s)\tilde{Q}D_su,D_s u\rangle+\sum_{i=1}^4\Re \left(\langle E_i u , u\rangle \right)\\ + \sum_{k_1+k_2=2}c_{2,k_1,k_2}\Re \left(\int_{s=0}\psi^{(k_1)}(0)(\prtl_s^{k_2}\tilde{Q})u\overline{u}\right)
\end{multline}
where
\begin{equation}\label{positiveterm}
E_1 u := (\psi'(s)\tilde{Q}+ 2\psi(s)(\prtl_s\tilde{Q}))(D_{x_1}D_{x_2} + s \tilde{R})u  +Q\tilde{R}u ,
\end{equation}
\begin{equation}\label{thirdderivterm}
E_2 u := \sum_{k_1+k_2=3}c_{3,k_1,k_2} \psi^{(k_1)}(s)\prtl_s^{k_2}\tilde{Q}u,
\end{equation}
\begin{equation}\label{energyterm}
E_3 u := sQ (\prtl_s\tilde{R})u +\frac 12 [(\prtl_s Q),D_{x_1}D_{x_2}+s\tilde{R}]u +2 (\prtl_s Q)Pu,
\end{equation}
\begin{equation}
E_4 u := [Q,D_{x_1} D_{x_2}+s\tilde{R}]\prtl_s u.
\end{equation}

The key lower bound on \eqref{commutator} follows from:
\begin{lemma}\label{thm:positivebound}
Given $E_1$, $E_2$ as defined in \eqref{positiveterm}, \eqref{thirdderivterm} we have that
\begin{equation}\label{positivebound}
\int_{s \geq 0 } \frac{\alpha(s)}{(C_1 \lambda^{-\frac 23} + s)^{\frac 12}}|D_{x_2} u(s,x)|^2\,dsdx
\lesssim  \Re \langle E_1 u, u \rangle + \Re \langle E_2 u, u \rangle + \|Du\|_{L^2}^2.
\end{equation}
\end{lemma}
Before proving the lemma, we discuss how it concludes the proof of \eqref{alphahalfspace}.  We first claim that
\begin{equation}\label{e3bound}
\sum_{k_1 + k_2 = 2}\left|\int_{s=0}\psi^{(k_1)}(0)\prtl_s^{k_2}\tilde{Q}u\overline{u}\right|+\left| \langle E_3 u , u\rangle \right| \lesssim \|Du\|_{L^2}^2 + \|Pu\|_{L^2}^2.
\end{equation}
The boundary trace is estimated using \eqref{secondsidebdry}.  Since $\prtl_sq \in S^{1}_{1/3,0}$, we have
\[
|\langle(\prtl_s Q)Pu,u\rangle|=|\langle Pu,(\prtl_s Q)u\rangle| \lesssim \|Pu\|_{L^2}\|Du\|_{L^2}
\]
which bounds the last term in \eqref{energyterm}. Next we observe that the symbolic calculus and \eqref{qregularity}, \eqref{qregimproved} means that the remaining operators in $E_3$ are in $\text{Op}(S^2_{1/3,0})$, which handles these terms by frequency localization.  In particular, the symbol of $[(\prtl_s Q),D_{x_1}D_{x_2}  ]$ does not involve any derivatives of the form $\prtl_{\xi_1}\prtl_s q$.

We next turn to the term $\langle E_4u, u \rangle$, claiming that
\begin{equation}\label{e4bound}
\left| \langle E_4 u , u\rangle \right| \lesssim \|\tilde{Q}\prtl_s u\|_{L^2}\|Du\|_{L^2}+\|Du\|_{L^2}^2.
\end{equation}
But this follows from the symbolic calculus, which allows us to write
\[
E_4 = [Q,D_{x_1} D_{x_2}+s\tilde{R}] = B_1\tilde{Q} +B_2
\]
for some $B_1,B_2 \in \text{Op}(S^{1}_{1/3,0})$ for each $s$, with symbol bounds uniform in $s$.

Now define $a(s,x,\xi) = (\tilde{q}(s,x,\xi))^{1/2} \in S_{1/3,0}^{1/6}(\RR^{2n}_{x,\xi})$ and let $A$ be the operator defined by the Fourier integral in \eqref{fourierintegral}. Thus for each $s$, $\tilde{Q}-A^2 \in \text{Op}(S_{1/3,0}^{0})$ and since functions of $s$ commute with $A$, we now have the following lower bound on the first term in \eqref{positivebound}
\begin{equation}\label{normallowerbound}
\int_{s\geq 0}\frac{\alpha(s)}{C_1\lambda^{-\frac 23}+s}AD_su\overline{AD_s u}\,dsdx \lesssim \Re\langle\psi'(s)\tilde{Q}D_su,D_s u\rangle + \|Du\|_{L^2}^2.
\end{equation}
We now recall the arguments at the end of \S\ref{ss:symplectic}.  Given the observations there, the lemma, \eqref{normallowerbound}, and the upper bounds above, we have the following estimate for any $\epsilon >0$, which we will see is an acceptable deviation from the one in \eqref{alphahalfspace}:
\begin{align*}
 \max_{1\leq j\leq J_\lambda} 2^{j}\|AD_s u\|_{L^2(\tilde{S}_j)}^2 &+\max_{1\leq j\leq J_\lambda} 2^{\frac j2}\|D_{x_2} u\|_{L^2(\tilde{S}_j)}^2   \\
 &\lesssim \|Du\|_{L^2}^2 + \|Pu\|_{L^2}^2 +\|\tilde{Q}D_su\|_{L^2}(\|Du\|_{L^2} + \|Pu\|_{L^2})\\
&\lesssim \epsilon\|\tilde{Q}D_su\|_{L^2}^2+(1+\epsilon^{-1})\left( \|Du\|_{L^2}^2 + \|Pu\|_{L^2}^2 \right).
\end{align*}
Taking $\epsilon>0$ sufficiently small, it suffices to show that
\begin{equation}\label{normalupperbound}
 \|\tilde{Q}D_su\|_{L^2}^2 \lesssim \max_{1\leq j\leq J_\lambda} 2^{j}\|AD_s u\|_{L^2(\tilde{S}_j)}^2 + \|Du\|_{L^2}^2.
\end{equation}
However, since $\tilde{Q}-A^2 \in \text{Op}(S_{1/3,0}^{0})$ for each $s$, it suffices to show that
\begin{equation}\label{asquared}
\|A^2D_s u\|_{L^2(\tilde{S}_j)}^2 \lesssim 2^{\frac j2}\|AD_s u\|_{L^2(\tilde{S}_j)}^2
\end{equation}
as geometric summation can then be used to handle the sum over all $1 \leq j \leq J_\lambda$ on the left (recalling that when $j=J_\lambda$, we change $\tilde{S}_j$ to $\tilde{S}_{J_\lambda}$). But this follows from the fact that for each $s$, $(\lambda^{-2/3}+s)^{1/4}a(s,\cdot) \in S_{1/3,0}^{0}(\RR^{2n}_{x,\xi})$.

In summary, we obtain the following estimate which is stronger than \eqref{alphahalfspace}
\[
\max_{1\leq j\leq J_\lambda} 2^{\frac j4}\|Du\|_{L^2(\tilde{S}_j)} + \max_{1\leq j\leq J_\lambda} 2^{\frac j2}\|AD_s u\|_{L^2(\tilde{S}_j)} \lesssim \|Du\|_{L^2} + \|Pu\|_{L^2}.
\]
Indeed, for each $s$, $1/a(s,x,\xi) \in S^{0}_{1/3,0}(\RR^{2n}_{x,\xi})$ so by the symbolic calculus, the second term on the left here dominates the second term on the left in \eqref{maincoordestimate} (up to acceptable error) $\|D_s u\|_{L^2(S_j)} \lesssim \|AD_s u\|_{L^2(\tilde{S}_j)} + \|Du\|_{L^2}$.

\begin{proof}[Proof of Lemma \ref{thm:positivebound}]
Begin by observing that up to acceptable errors in $S^2_{1/3,0}(\RR_{x,\xi}^{2n})$, the symbol of $E_1$ can be computed as
\begin{multline}\label{poissonsimp}
\frac{\alpha(s)\xi_2^2\beta_\lambda(\xi)(\frac{\xi_1}{\xi_2}+s r(s,x,\xi))}{(C_1\lambda^{-\frac 23} + \frac{\xi_1}{\xi_2} \zeta r^{-1}  + s)^{\frac 12} (C_1\lambda^{-\frac 23} + s) }+\\
\frac{\psi(s)\xi_2^2\beta_\lambda(\xi)\left[ C_1\lambda^{-\frac 23}r +\left(-\frac{\xi_1}{\xi_2}\right)(1-\zeta)+ \left(-\frac{\xi_1}{\xi_2}\right) \left(\frac{\xi_1}{\xi_2}+sr\right)r^{-1}\prtl_{s}\zeta \right]}{(C_1\lambda^{-\frac 23}+\frac{\xi_1}{\xi_2} \zeta r^{-1}  + s)^{\frac 32}}
\end{multline}
Strictly speaking, we should include a term of the form $\left(-\frac{\xi_1}{\xi_2}\right) \left(\frac{\xi_1}{\xi_2}+sr\right)\zeta\prtl_{s}r^{-1}$ in the brackets and should also account for derivatives of $\beta_\lambda(\xi)$, but these terms can be neglected as their contribution is in $S^2_{1/3,0}$ for every $s$. Our first task is now to check that \eqref{poissonsimp} is bounded below by
\begin{equation}\label{pcmlower}
\tilde{\delta}\xi_2^2\frac{\alpha(s)}{(C_1\lambda^{-\frac 23} + s)^{\frac 12} }
\end{equation}
for some sufficiently small constant $\tilde{\delta} >0$.  Later on, we will see that \eqref{poissonsimp} also dominates the contribution of $ \Re \langle E_2 u, u\rangle $.

We first treat the term involving $\prtl_{s}\zeta$, observing that
\[
\prtl_{s}\zeta = \left(-\frac{\xi_1}{\xi_2}\right)\left(\frac{1}{s^2}\right)\tilde{\zeta}'\left(\frac{\xi_1}{\xi_2}s^{-1}\right)\chi(\lambda^{\frac 23}s)+\lambda^{\frac 23}\chi'(\lambda^{\frac 23}s)\left(\tilde{\zeta}\left(\frac{\xi_1}{\xi_2}s^{-1}\right)- \chi(\lambda^{-\frac 13}\xi_1)\right)
\]
The contribution of the first term to the brackets in \eqref{poissonsimp} here is nonnegative since $\tilde{\zeta}$ is increasing and $\xi_1\xi_2^{-1} + sr>0$ on the support of that term.  The second is only nonzero when $1 \leq \lambda^{\frac 23} s \leq 2$, $\chi(\lambda^{-\frac 13}\xi_1)=0$ and $\tilde{\zeta}(\xi_1\xi_2^{-1}s^{-1})>0$, that is,
\[
\xi_1 \lesssim \lambda^{\frac 13} \quad \text{and} \quad -\xi_1 \leq \frac{c_0}2 \xi_2 s \approx \lambda^{\frac 13}.
\]
Since $\left(-\frac{\xi_1}{\xi_2}\right)\left(\frac{\xi_1}{\xi_2}+sr\right) \approx \lambda^{-\frac 43}$ for such $s,\xi_1,\xi_2$, the contribution of this term in \eqref{poissonsimp} can be dominated by the term $C_1 \lambda^{-\frac 23} r$ by taking $C_1$ sufficiently large.

In the region $\zeta >0 $, the remaining portion of the second term in \eqref{poissonsimp} is nonnegative, by choosing $C_1$ large, and the first is bounded below by
\[
\frac{\alpha(s)\xi_2^2\left(\frac{\xi_1}{\xi_2}+sr(s,x,\xi)\right)^{\frac 12}}{C_1\lambda^{-\frac 23}  + s } \gtrsim \frac{\alpha(s)\xi_2^2}{(C_1\lambda^{-\frac 23} + s)^{\frac 12} }.
\]
We are left to consider the case when $\zeta\equiv 0$.  In this case we need to see that
\[
\alpha(s)\left(\frac{|\xi_1|}{\xi_2}+s\right) \ll \psi(s) \left(C_1\lambda^{-\frac 23}+\left(-\frac{\xi_1}{\xi_2}\right)\right)
\]
When $s \leq \lambda^{-\frac 23}$ this is clear in the region $|\xi_1| \leq \frac{C_1}{100}\lambda^{-\frac 13}$ as $\alpha(s)=1$ here.  In all other cases, $-\frac{\xi_1}{\xi_2} \geq \frac{c_0}2 s$ and the right hand side dominates the left hand side by taking $C_2$ sufficiently large in the definition of $\psi$ (say $C_2\geq 100$).

We now turn to the contribution of $E_2$, and we claim that the absolute value of this is majorized by the expression in \eqref{poissonsimp}.  A tedious computation reveals that when $3 \geq k\geq 1$, $\prtl_{s}^k \zeta$ is supported where $s \geq \lambda^{-\frac 23}$ and that $|\prtl_{s}^k \zeta| \lesssim s^{-k}$.  Moreover, it can be seen that for these $k$,
\[
\left|\frac{\xi_1}{\xi_2}\prtl_{s}^k\zeta \right| \lesssim s^{1-k}.
\]
Hence another computation reveals that
\begin{equation}\label{thirdderivsymbol}
\left|\psi^{(k_1)}(s)\prtl_s^{k_2}\tilde{q}\right| \lesssim C_2(C_1\lambda^{-\frac 23} + s)^{-\frac 32}(\lambda^{-\frac 23} + s)^{-2}.
\end{equation}
Indeed, the worst possible contribution comes from the case where $k_2=3$ and the term which involves $\prtl_{s}^3\zeta$.  Otherwise, one has larger powers of $(C_1\lambda^{-\frac 23} + s)^{-1}$.

First consider the more difficult $\zeta >0$ region.  Here we need to see that we may dominate the right hand side of \eqref{thirdderivsymbol} by the first term in \eqref{poissonsimp}
\[
C_2(C_1\lambda^{-\frac 23} + s)^{-\frac 32}(\lambda^{-\frac 23} + s)^{-2} \ll \frac{\alpha(s)(\frac{\xi_1}{\xi_2}+s)^{\frac 12}\lambda^2}{C_1\lambda^{-\frac 23}+s}
\]
or equivalently (after cancelation and multiplying both sides by $\lambda^{-\frac 53}$)
\[
C_2(C_1 + \lambda^{\frac 23}s)^{-\frac 12}(1+ \lambda^{\frac 23}s)^{-2} \ll \alpha(s)\left(\lambda^{\frac 23}\frac{\xi_1}{\xi_2}+\lambda^{\frac 23}s\right)^{\frac 12}
\]
Since we may assume that $\alpha(s) = 1$ for $s \leq \lambda^{-\frac 23}$, this inequality is clear for such $s$ by taking $C_2 \ll C_1$, since $\lambda^{2/3}\xi_1/\xi_2\geq 1$ in this case.  Otherwise, we need to exploit some facts about our function $\alpha(s)$.  Given our definition of slowly varying \eqref{slowvary}, we may assume that for $s \geq 0$
\begin{equation}\label{alphafcnvary}
2^{-3\delta}\leq \frac{\alpha(2s)}{\alpha(s)}\leq 2^{3\delta},  \qquad \text{and for $k \geq 1$,} \qquad 2^{-3\delta k} \leq \frac{\alpha(2^k s)}{\alpha(s)} \leq 2^{3\delta k},
\end{equation}
the latter being a consequence of the former and induction. Since $\alpha(s) = 1$, we can take dyadic numbers $k_1,k_2$ such that $2^{-k_1} \approx \lambda^{-\frac 23}$ and $2^{-k_2}\approx s$, to see that
\[
\alpha(s) \geq (\lambda^{\frac 23}s)^{-3\delta}2^{-6\delta}.
\]
For $s \geq \lambda^{-\frac 23}$ it now suffices to show that
\[
C_2(C_1 + \lambda^{\frac 23}s)^{-\frac 12}(\lambda^{\frac 23}s)^{-2} \ll (\lambda^{\frac 23}s)^{\frac 12-3\delta}
\]
But this again is clear by taking $C_1 \gg C_2^{4}$ and $3\delta \ll \frac 12$.

The $\zeta =0$ case is easier as here it is sufficient to see that the right hand side of \eqref{thirdderivsymbol} is dominated by the second term in \eqref{poissonsimp}, leading us to bound
\[
(C_1\lambda^{-\frac 23} + s)^{-\frac 32}(\lambda^{-\frac 23} + s)^{-2} \ll (C_1\lambda^{-\frac 23} + s)^{-\frac 32}\lambda^2(C_1 \lambda^{-\frac 23})
\]
after canceling $\psi$.  But this is equivalent to $(\lambda^{-\frac 23} + s)^{-2} \ll C_1\lambda^{\frac 43}$, which can be arranged by taking $C_1$ large.

We now define a new symbol $b(s,x,\xi)$ by
\[
\xi_2^2 b(s,x,\xi)= \eqref{poissonsimp} + \sum_{k_1+k_2=3}c_{3,k_1,k_2} \psi^{(k_1)}(s)\prtl_s^{k_2}\tilde{q}(s,x,\xi) - \eqref{pcmlower},
\]
for some $\tilde{\delta}>0$ sufficiently small in \eqref{pcmlower}.  Defining the operator $B$ as in \eqref{fourierintegral}, we are reduced to seeing that there exists $C$ sufficiently large such that for any $s$,
\begin{equation}\label{feffphongclaim}
0 \leq \Re\left( \int BD_{x_2}u(s,x)\overline{D_{x_2}u(s,x)}\,dx\right)
+ C\|D_{x_2}u(s,\cdot)\|_{L^2(\RR^n)}^2. 
\end{equation}
Indeed, the symbolic calculus shows that the difference between the left hand side of \eqref{positivebound} and the first term here is dominated by the acceptable error $\lambda^2\|u\|_{L^2}^2 \approx \|Du\|_{L^2}^2$.  To see \eqref{feffphongclaim}, we make a comparatively crude argument which uses a variant of the Fefferman-Phong inequality (see \cite[Theorem 0.7.C]{tay91}), stating that the inequality \eqref{feffphongclaim} holds provided $b(s,\cdot) \in S_{\rho,\delta}^{2(\rho-\delta)}$ (uniformly and with $\rho >\delta$).

First observe that
\begin{equation}\label{bapprox}
b(s,x,\xi) \approx \begin{cases}
\alpha(s)\left(\frac{\xi_1}{\xi_2}+s r(s,x,\xi)\right)^{\frac 12}(C_1\lambda^{-2/3} + s)^{-1}, & \zeta >0\\
(C_1\lambda^{-\frac 23}-\frac{\xi_1}{\xi_2})(C_1\lambda^{-2/3} + s)^{-3/2}, & \zeta = 0
\end{cases}.
\end{equation}
As above, the derivatives of $b$ in $\xi_1$ are the most poorly behaved, gaining powers of $\left(C_1\lambda^{-\frac 23}\xi_2 + \xi_1r^{-1}\zeta+\xi_2 s \right)^{-1}$, while derivatives in the remaining directions $\xi_2,\dots,\xi_n$ gain powers of $\langle \xi \rangle^{-1}\approx \lambda^{-1}$.  Hence naively, one has $b(s,\cdot) \in S^1_{1/3,0}$, and the Fefferman-Phong inequality does not apply.  However, when $\lambda^{1/2} s$ is bounded below by a small number, the derivatives in $ \xi_1$ gain powers of $\lambda^{-1/2}$ and the symbol is in $S^1_{1/2,0}$, correcting this matter.  Otherwise, when $s \ll \lambda^{-1/2}$, we use smooth bump functions to write $b=b_1+b_2$ where $\supp(b_1) \subset \{\xi_1 \leq -\lambda^{1/2}\}$ and $\supp(b_2) \subset \{\xi_1 \geq -2\lambda^{1/2}\}$.  In this case we have that $\supp(b_1) \cap \supp(\zeta) = \emptyset$, so over $\supp(b_1)$, the derivatives of $b$ in $\xi_1$ gain full powers of $(1+|\xi|)^{-1}\approx \lambda^{-1}$, so we may select the bump function so that $b_1 \in S^1_{1/2, 0}$.  On the other hand, using \eqref{bapprox}, we have $b_2 \lesssim \lambda^{\frac 23}$, so $b_2 \in S_{1/3, 0}^{2/3}$.  Thus applying the Fefferman-Phong inequality separately to the operators defined by $b_1$ and $b_2$ yields the desired bound.
\end{proof}

\section{Strichartz and square function estimates}\label{sec:stz}
\subsection{Preliminary reductions}\label{ss:stzprelim}
In this section we prove Theorems \ref{thm:sqfcn} and \ref{thm:wavestz}.  We revert back to the notation in \S\ref{ss:coordintro}, letting $x_0$ denote the time coordinate and $x_n$ a defining function for the boundary. Recall that given the discussion following Theorem \ref{thm:wavestz}, it suffices to prove \eqref{coordstz}, \eqref{sqfcncoord} for $u$ compactly supported in $x_0$.  We extend $u$ to all of $\RR^{n+1}$ by reflecting the solution and the initial data the in boundary $x_n=0$ in an even or odd fashion corresponding to Neumann or Dirichlet conditions, thus preserving the boundary condition $u|_{x_n}=0$ or $\prtl_{x_n}u|=0$.  Extending the metric coefficients in an even fashion $g^{ij}(x',|x_n|)$, $Pu$ will then have the same parity as $u$ over all of $\RR^{n+1}$.  The extended coefficients are thus Lipschitz and satisfy \eqref{c0}.

In this section, the preliminary reductions are common to both Strichartz and square functions estimates.  To enable us to treat them simultaneously, we let $X$ denote the function space corresponding to either Strichartz or square function bounds:
\[
\|u\|_X = \left( \int \left(\int|u(x)|^q dx' \right)^{\frac pq}dx_0\right)^{\frac 1p}
\text{ or } \|u\|_X = \left( \int \left(\int |u(x)|^2 dx_0 \right)^{\frac q2}dx'\right)^{\frac 1q}.
\]
where $x'$ denotes $(x_1,\dots,x_n)$ and the domain of integration in this variable is over all of $\RR^n$.  Since we assume that $u$ is compactly supported in $x_0$, we may take the domain of integration in this variable to be $(-\infty,\infty)$.

Next we observe that it suffices to show that the extended solution $u$ satisfies
\begin{equation}\label{firstcoordred}
\|\langle D \rangle^{1-\gamma} u \|_{X} \lesssim \|u\|_{H^1(\RR^{n+1})}+\|Pu\|_{L^2(\RR^{n+1})}.
\end{equation}
In other words, it suffices to assume that the Sobolev spaces (respectively vector valued Sobolev spaces in \eqref{sqfcncoord}) can be replaced by the usual Sobolev spaces over $\RR^n$ defined with respect to the Fourier multiplier $\langle \xi \rangle = (1+|\xi|^2)^{1/2}$.  But this follows from interpolation since $\|f\|_{W^{2,q}(\RR^{n}_+)} = \|\tilde{f}\|_{W^{2,q}(\RR^n)}$ whenever $\tilde{f}$ is an odd (respectively even) extension of a function satisfying $\tilde{f}|_{x_n=0}=0$ (respectively $\prtl_{x_n}\tilde{f}|_{x_n=0}=0$), and similarly for the vector valued counterpart.

We now note that it suffices to assume that the support of the space time Fourier transform $\supp(\widehat{u})$ is supported away from the origin.  This follows from Sobolev embedding and the fact that the commutator of $P$ with a smooth cutoff to $|(\xi_0,\dots,\xi_1)|\lesssim 1$ will map $L^2 \to H^1$.  Next let $\Lambda$ be a Fourier multiplier defined by a homogeneous function of degree zero such that $\supp(\Lambda)\subset\{|\xi_0| \approx |(\xi_1,\dots,\xi_n)|  \}$ so that $\supp(1-\Lambda)$ is disjoint from the characteristic set of $P$; this can be arranged by taking $\veps$ sufficiently small in \eqref{c0}.  Sobolev embedding and elliptic regularity (see for example, \cite[Theorem 2.2.B]{tay91}) then give
\[
 \|\langle D \rangle^{1-\gamma}(1-\Lambda) u\|_X \lesssim  \|(1-\Lambda) u\|_{H^{\frac 32-\gamma}} \lesssim   \|u\|_{H^1(\RR^{n+1})}+\|Pu\|_{ L^2(\RR^{n+1})}.
\]
Here we have implicitly used that $[P,\Lambda]:H^1 \to L^2$, which follows since the Coifman-Meyer commutator theorem guarantees $[g^{ij},\Lambda]:L^2\to H^1$.  Note that $\supp(\Lambda)$ splits into two disjoint regions corresponding to $\pm \xi_0 >0$.  Taking a finer partition of unity, we may further restrict the support of $\Lambda$, assuming that it is supported where $\xi_0>0$ and where $\xi_1 \gg |(\xi_2,\dots,\xi_n)|$ as a rotation of coordinates treats the other cases.

We take a careful Littlewood-Paley decomposition. Let $\{\beta_l(\zeta)\}_{l=0}^\infty$ be a sequence of smooth functions $\beta_l:[0,\infty) \to [0,1]$ such that
\begin{equation}\label{littpaleyseq}
\sum_{l=0}^\infty \beta_l(\zeta) =1 \text{ for } \zeta \geq 0, \qquad \beta_l(\zeta) = \beta_1(2^{-l+1}\zeta) \text{ for } l\geq 1,
\end{equation}
with $\supp(\beta_0) \subset [0,2)$ and $\supp(\beta_1) \subset (2^{-\frac 12},2^{\frac 32})$.  Now let $\beta_k^T$ be the Fourier multiplier with symbol $\beta_k(|\xi'|)$ (the ``$T$'' here  signifying ``tangential").

Next, let $\beta_l^N$ denote the Fourier multiplier with symbol $\beta_l(|\xi_n|)$ for $l > k$.  We then let $\beta^N_{<k}$ denote the multiplier with symbol $1-\sum_{k <l} \beta_l(|\xi_n|)$ so that the multiplier truncates to frequencies $\lesssim 2^{k}$ and
\begin{equation}\label{littpaleylambda}
\langle D \rangle^{1-\gamma} \Lambda u = \sum_{k=1}^\infty \beta^N_{<k}\beta^T_k (\langle D \rangle^{1-\gamma} \Lambda u) + \sum_{k=1}^\infty \sum_{l=k+1}^\infty \beta^N_{l}\beta^T_k (\langle D \rangle^{1-\gamma} \Lambda u).
\end{equation}
Applying the Littlewood-Paley square function estimate first in $l$, then in $k$ yields
\[
\left\| \sum_{k=1}^\infty \sum_{l=k+1}^\infty \beta^N_{l}\beta^T_k (\langle D \rangle^{1-\gamma} \Lambda u)\right\|_X^2 \lesssim \sum_{k=1}^\infty \sum_{l=k+1}^\infty 2^{l(1-\gamma)}\left\| \beta^N_{l}\beta^T_k ( \Lambda u)\right\|_X^2.
\]
It is clear that this can be done for Strichartz estimates, and in the case of square function estimates, we use that $\langle D \rangle^{1-\gamma} \Lambda$ localizes to $\xi_0\approx \xi_1$ to apply the Littlewood-Paley estimate in the $x_0$ variable.  Similarly, we have
\[
\left\|\sum_{k=1}^\infty \beta^N_{<k}\beta^T_k (\langle D \rangle^{1-\gamma} \Lambda u)\right\|_X^2 \lesssim \sum_{k=1}^\infty 2^{k(1-\gamma)}\|\beta^N_{<k}\beta^T_k ( \Lambda u)\|_X^2 .
\]

We now show that we may bound the second sum in \eqref{littpaleylambda}.  Let $P_l$ denote the differential operator obtained by truncating the metric coefficients in the frequency variable to frequencies $|(\xi_0,\xi_1,\dots,\xi_n)|\lesssim 2^l$.  Note that if $g^{ij}_l$ denotes the regularization of $g^{ij}$, then $\|g^{ij}_l-g^{ij} \|_{L^\infty}\lesssim 2^{-l}$.  The angle one parametrices from \cite{smithsogge07} and \cite{bssbdrystz}, which amounts to the $\theta =1$ case in \eqref{smsosqfcn} below, show that
\begin{equation}\label{angleonebound}
2^{l(1-\gamma)}\|\beta^N_{l}\beta^T_k \Lambda u\|_X \lesssim 2^{l}\|\beta^N_{l}\beta^T_k \Lambda u\|_{L^2} + \|P_l(\beta^N_{l}\beta^T_k \Lambda u)\|_{L^2}.
\end{equation}
We have that
\begin{equation*}
 \sum_{l\geq k}^\infty \left(2^{l}\|\beta^N_{l}\beta^T_k \Lambda u\|_{L^2} + \|(P-P_l)(\beta^N_{l}\beta^T_k \Lambda u)\|_{L^2}+\|\beta^N_{l}\beta^T_k P(\Lambda u)\|_{L^2}\right)
\end{equation*}
is bounded by the right hand side of \eqref{firstcoordred}, which uses the consequence of the Coifman-Meyer commutator bound $[P,\Lambda]:L^2\to H^1$ (cf. the ensuing argument).  It thus suffices to see that
\[
\sum_{k=1}^\infty \sum_{l=k}^\infty \left\|[P,\beta^N_{l}\beta^T_k] \Lambda u\right\|_{L^2}^2
\]
is bounded by the right hand side of \eqref{firstcoordred}.  Khinchin's inequality reduces this to showing that for an arbitrary sequence $\epsilon_{l,k}=\pm 1$
\begin{equation}\label{khinchin}
\left\|\left[g^{ij},\sum_{k=1}^\infty \sum_{l=k}^\infty \epsilon_{l,k}\beta^N_{l}\beta^T_k\right] v \right\|_{L^2} \lesssim \|v\|_{H^{1}},
\end{equation}
But this a consequence of bootstrapping the Coifman-Meyer commutator theorem.

A similar line of reasoning can be applied to the first sum in \eqref{littpaleylambda}, hence it suffices to show that $\beta^N_{<k}\beta^T_k (\Lambda u)$ satisfies the following estimates akin to \eqref{angleonebound}
\begin{equation*}
2^{k(1-\gamma)}\|\beta^N_{<k}\beta^T_k \Lambda u\|_X \lesssim 2^{k}\|\beta^N_{<k}\beta^T_k \Lambda u\|_{L^2} + \|P_l(\beta^N_{<k}\beta^T_k \Lambda u)\|_{L^2}.
\end{equation*}
Note that without loss of generality, we may assume that $\Lambda(\xi)$ is independent of $\xi_n$ near $\xi_n=0$ so that $(\beta^N_{<k}\Lambda)(\xi)$ can be written as $\beta^N_{<k}(\xi_n)\Lambda(\xi_0,\dots,\xi_{n-1})$.  Since $\beta^N_{<k}(\xi_n)$ is even, $\beta^N_{<k}\beta^T_k (\Lambda u)$ satisfies the same homogeneous boundary conditions as $u$.  Relabeling this function as $u_\lambda$, $\lambda = 2^k$, the desired estimates now result from:
\begin{theorem}\label{thm:coordstz}
Suppose in the coordinate system chosen above, the Fourier support of $u_\lambda$ is localized to a conic region $\xi_1 \gg |(\xi_2,\dots,\xi_n)|$ and also a region where $\xi_0\approx \xi_1 \approx \lambda$.  Assume further that $P_\lambda$ is the operator formed by truncating the coefficients of the operator in \eqref{Pstz} to frequencies less than $\lambda$. If $u_\lambda$ satisifies boundary conditions $u|_{x_n=0}=0$ or $\prtl_{x_n}u|_{x_n=0}$, then
\begin{equation}\label{stzreduction}
\|u_\lambda\|_X \lesssim \lambda^{\gamma}\|u_\lambda\|_{L^2} + \lambda^{\gamma-1}\|P_\lambda(u_\lambda)\|_{L^2}.
\end{equation}
\end{theorem}

\subsection{The nontangential/tangential decomposition} Here we introduce the decomposition used in \cite{blairrls} which allows us to use the localized energy estimates to bound the error terms in a wave packet parametrix for $P$, thus yielding Theorem \ref{thm:coordstz}. Begin by defining $\sigma$ as a function of $q$, $p$ by
\begin{equation}\label{sigmadef}
 \sigma =
\begin{cases}
(n-1)(\frac 12-\frac 1q)-\frac 2p, & X = L^p L^q
\\
(n-1)(\frac 12-\frac 1q)-\frac 2q, & X = L^q L^2
\end{cases}.
\end{equation}
The motivation for this choice is that it characterizes the gain $\theta_j^{\sigma}$ in the $X$ estimates for solutions which are localized to a cone $|\xi_n|\lesssim \lambda\theta_j$ established in \cite{smithsogge07}, \cite{bssbdrystz}.  Note that the subcritical hypotheses on the exponents defining $X$ ensure that $\sigma>0$.

We next choose $\alpha<2/3$ such that $\frac{1}{3\alpha}-\frac 12 < \sigma$.  In this section, we let $J_\alpha$ be the largest integer such that $2^{-J_\alpha} \geq \lambda^{-\alpha}$.  For $1 \leq j \leq J_\alpha$, let $\chi_j(x_n) = \chi(2^j x_n)$ where $\supp(\chi)\subset (-2,2)$ and $\supp(1-\chi)\subset \RR\setminus [-1,1]$.  Then let $\psi_j=\chi_j-\chi_{j+1}$ so that $\supp(\psi_j) \subset \{2^{-j} \leq |x_n| \leq 2^{1-j}\}$.  Now take a sequence of smooth cutoffs to be applied in the $\xi_n$ variable, such that $\supp(\Gamma_j) \subset \{|\xi_n| \approx \lambda 2^{-\frac j2} \}$, $\supp(\Gamma_{<j}) \subset \{|\xi_n| \lesssim \lambda 2^{-\frac j2}\}$,\footnote{Since the ``$<j$" in the subscript here denotes that the support lies in $\{|\xi_n| \lesssim \lambda 2^{-\frac j2}\}$, this is a slight deviation in the notational convention above.}   and
\[
\Gamma_{<j}(\xi_n) = \Gamma_{<L}(\xi_n)+\sum_{l=j}^{L-1} \Gamma_l(\xi_n), \qquad 1 = \Gamma_{<J_\alpha}(\xi_n)+\sum_{1 \leq j < J_\alpha} \Gamma_j(\xi_n)
\]
for $\xi_n$ in the projection of the support of $\widehat{u} $.  Now define
\begin{align*}
 w_j &= \Gamma_j(D_{x_n})(\chi_j u_\lambda), & 1 \leq j < J_\alpha \\
 v_j &= \Gamma_{<j}(D_{x_n})(\psi_j u_\lambda)=\left(\Gamma_{<J_\alpha}(D_{x_n})+\sum_{l=j}^{J_\alpha-1} \Gamma_l(D_{x_n}) \right)\psi_j u_\lambda , & 1 \leq j < J_\alpha \\
v_{J_\alpha} &= \Gamma_{<J_\alpha}(D_{x_n})(\chi_{J_\alpha}u_\lambda).
\end{align*}
It is not hard to verify (cf. \cite[p.792-3]{blairrls}) that $u_\lambda = v_{J_\alpha} + \sum_{j=1}^{J_\alpha-1}(v_j+w_j)$.\footnote{The purpose of the decomposition is so that $\cup_j \supp(w_j)$ captures the nontangential reflections of $u_\lambda$ in the boundary.  In the process, $\cup_j \supp(v_j)$ contains bicharacteristic rays which are distant from the boundary, making ``tangential" a slight misnomer.}

The main idea in showing Theorem \ref{thm:coordstz} is that the results in \cite{smithsogge07}, \cite{bssbdrystz} imply that when $1 \leq j < J_\alpha$
\begin{align}
 \|w_j\|_X &\lesssim \lambda^{\gamma}2^{-\frac{j\sigma}{2}}\left( 2^{\frac j4}\|w_j\|_{L^2} +  \lambda^{-1}2^{-\frac j4}\|P_\lambda(w_j)\|_{L^2}\right),\label{wjbound}\\
 \|v_j\|_X &\lesssim \lambda^{\gamma}2^{-\frac{j\sigma}{2}}\left( 2^{\frac j4}\|v_j\|_{L^2} +
 \lambda^{\frac 12}2^{-\frac j8}   
 \|\langle \lambda^{\frac 12}2^{-\frac j4}x_n\rangle^{-M}v_j\|_{L^2}+ \lambda^{-1}2^{-\frac j4}\|P_\lambda(v_j)\|_{L^2}\right),\label{vjbound}
\end{align}
and that
\begin{equation}\label{vJbound}
  \|v_{J_\alpha}\|_X \lesssim \lambda^{\gamma}2^{-\frac{{J_\alpha}\sigma}{2}}\left( \lambda^{\frac 16}\|v_{J_\alpha}\|_{L^2} +  \lambda^{-\frac 76}\|P_\lambda(v_{J_\alpha})\|_{L^2}\right).
\end{equation}
Postponing the proof of these estimates, we show that they yield Theorem \ref{thm:coordstz}.  The argument is essentially the same as the one in \cite[\S2.3]{blairrls}, but we review it here for the sake of completeness.
We will take $M$ sufficiently large based on $\alpha$ in the middle term on the right in \eqref{vjbound}; its presence will be motivated later on.  For now we note that since $\Gamma_{<j}(D_{x_n})$ is rapid decaying outside of $\lambda 2^{-j/2}$ neighborhood of $\{2^{-j} \leq |x_n| \leq 2^{-j+1}\}$ and $\lambda 2^{-\frac{3j}{2}} \geq \lambda^{1-\frac{3\alpha}{2}} \gg 1$, we have
\[
 \|\langle \lambda^{\frac 12}2^{-\frac j4}x_n\rangle^{-M}v_j\|_{L^2} \lesssim ( \lambda 2^{-\frac{3j}{2}})^{-\frac M2}\|u_\lambda\|_{L^2}\lesssim (\lambda^{1-\frac{3\alpha}{2}})^{-\frac M2}\|u_\lambda\|_{L^2}
\]
so that the factor on the right can be made smaller than $\lambda^{-1/2}$ by choosing $M$ large.  Next, we will show that for all choices of $j$, including $j=J_\alpha$,
\begin{equation}\label{locsmoothingvj}
  2^{\frac j4}\|v_j\|_{L^2} +  \lambda^{-1}2^{-\frac j4}\|P_\lambda(v_j)\|_{L^2} \lesssim \|u_\lambda\|_{L^2} +
\|P_\lambda(u_\lambda)\|_{L^2},
\end{equation}
and that the same holds when $v_j$ is replaced by a $w_j$. Therefore we may use the geometric gains of $2^{-j\sigma/2}$ in \eqref{vjbound}, \eqref{wjbound} to see that $\sum_{j=1}^{J_\alpha -1}(\|v_j\|_X+\|w_j\|_X)$ is dominated by the right hand side of \eqref{stzreduction}.  The term $v_{J_\alpha}$ is then estimated by using that the gain here, combined with the gain of $2^{-\frac{{J_\alpha}\sigma}{2}}$ in \eqref{vJbound} counterbalances the loss of $\lambda^{\frac 16}$.  Indeed, our choice of $\sigma$ ensures that $\lambda^{\frac 16}2^{-\frac{{J_\alpha}\sigma}{2}-\frac{{J_\alpha}}{4}}\approx \lambda^{\frac{1}{6}-\frac{\alpha}{4} - \frac{\alpha\sigma}{2}} \ll 1$.

Too see \eqref{locsmoothingvj}, when $1 \leq j < J_\alpha$, first observe that given the frequency localization, $\lambda^{-1}\|(P-P_\lambda)(u_\lambda)\|_{L^2}\lesssim \|u_\lambda\|_{L^2}$ since the coefficient smoothing yields a gain of $\lambda^{-1}$. Thus the localized energy estimates in Theorem \ref{thm:maincoordestimate} and frequency localization give
\begin{equation}\label{locsmoothingvj2}
 2^{\frac j4}\|v_j\|_{L^2} \lesssim \|u_\lambda\|_{L^2} + \lambda^{-1}\|P_\lambda(u_\lambda)\|_{L^2} .
\end{equation}
We are left to bound
\[
 P_\lambda(v_j) = [P_\lambda, \Gamma_{<j}]\psi_j u_\lambda + \Gamma_{<j} [P_\lambda,\psi_j] u_\lambda +\Gamma_{<j}\psi_j (P_\lambda u_\lambda).
\]
It is straightforward to bound the contribution the last term, so we will show that
\begin{equation}\label{microlocalcommutator}
\lambda^{-1}2^{-\frac j4}\|  P_\lambda(v_j) -\Gamma_{<j}\psi_j (P_\lambda u_\lambda) \|_{L^2} \lesssim 2^{\frac j4}\|\chi_{j-1} u_\lambda\|_{L^2}
\end{equation}
which in turn is bounded using the same argument as in \eqref{locsmoothingvj2}.  The commutator $[P_\lambda, \Gamma_{<j}]$ can be written as a sum of $[\Gamma_{<j},g^{lm}_\lambda]D_{x_l}D_{x_m}$ and $[\Gamma_{<j},D_{x_l}g^{lm}_\lambda]D_{x_m}$ where $l,m \neq n$ and $g^{lm}_\lambda$ denote the coefficients of $P_\lambda$.  Observe that
\[
\|[g^{lm}_\lambda,\Gamma_{<j}]\|_{L^2 \to L^2} \lesssim  \lambda^{-1} 2^{j/2}
\]
uniformly in $\lambda$, $j$. This can be verified by examining its Schwartz kernel since $g^{lm }_\lambda\in C^1$ and we may assume symbol bounds of the form $|\Gamma_{<j}^{(k)}|\lesssim (\lambda^{-1}2^{j/2})^k$ (cf. \cite[(2.30)]{blairrls}). Frequency localization thus bounds the contribution of this commutator.  As for the second commutator term
\begin{equation}\label{psicomm}
\Gamma_{<j} [P_\lambda,\psi_j] u_\lambda = -2\Gamma_{<j}\prtl_{x_n}\left((\prtl_{x_n}\psi_j) u_\lambda \right) + \Gamma_{<j}(\prtl^2_{x_n}\psi_j) u_\lambda.
\end{equation}
Since $|\prtl^k_{x_n}\psi_j|\lesssim 2^{jk}$ and $\lambda^{-1}2^{j/2}\Gamma_{<j}\prtl_{x_n}:L^2 \to L^2$ uniformly, we see that the contribution of this term is bounded above by the right hand side of \eqref{microlocalcommutator}.  The proof of \eqref{locsmoothingvj} when $j=J_\alpha$ or when $v_j$ is replaced by $w_j$ is similar.

\subsection{Square function bounds}\label{ss:sqfcncoord} We now use results in \cite{smithsogge07} (and later \cite{bssbdrystz}) to prove \eqref{wjbound}, \eqref{vjbound}, \eqref{vJbound}.
For now we assume $X=L^q L^2$, and discuss the case of Strichartz bounds in the next section.  In this case, it is convenient to use the microlocalization of $u_\lambda$ to treat $P_\lambda$ as an operator hyperbolic in $x_1$, taking the factorization
\[
g^{11}(x)\left( \xi_1 + q^+(x,\xi') \right)\left( \xi_1 - q^-(x,\xi') \right), \quad \xi'=(\xi_0,\xi_2,\dots,\xi_n),\, x'=(x_0,x_2,\dots,x_n),
\]
for some functions $q^\pm$ which are positive on the support of $w_j$, $v_j$.  The notation $x',\xi'$ thus plays a slightly different role here than previously.  We then let $q_\lambda(x,\xi')$ denote the symbol obtained by truncating $q^-$ to $x$ frequencies $\ll\lambda$ and use $Q_\lambda$ to denote the corresponding operator\footnote{The work \cite{smithsogge07} uses the notation $P_\lambda$ to denote this first order operator.}.  Moreover, since $\xi_1 + q_\lambda(x,\xi') >0$ on the support of  $w_j$, $v_j$, we have by elliptic regularity (cf. \cite[p.115]{smithsogge07})
\begin{align*}
\|(D_{x_1}-Q_\lambda(x,D_{x'}))v_j\|_{L^2} &\lesssim \lambda^{-1}\|P_\lambda(v_j)\|_{L^2}+\|v_j\|_{L^2},\\
\|(D_{x_1}-Q_\lambda(x,D_{x'}))w_j\|_{L^2} &\lesssim \lambda^{-1}\|P_\lambda(w_j)\|_{L^2}+\|w_j\|_{L^2}.
\end{align*}

Let $\theta_j:=2^{-\frac j2} \in [\lambda^{-\frac 13}, 1]$ and suppose $\Gamma_j$ is a microlocal cutoff to frequencies $|\xi_n| \approx \lambda 2^{-\frac j2}=\lambda \theta_j$.  For some sufficiently small $\tilde{\veps}>0$ we also define slabs
\[
S_{j,k}=[k\tilde{\veps}\theta_j,(k+1)\tilde{\veps}\theta_j]_{x_1} \times \RR^{n}_{x'} \qquad S=[-1,1]_{x_1} \times \RR^{n}_{x'}
\]
with $S_{j,k} \subset S$.  Given the results in \cite{smithsogge07}, we have the following estimate on arbitrary functions $u_\lambda$ such that $\widehat{u}_\lambda$ is supported in $\xi_0 \gg |(\xi_2,\dots,\xi_n)|$ (that is to say \emph{any} such function, not just the one which birthed $v_j$ and $w_j$) 
\begin{equation}\label{smsosqfcn}
 \|\Gamma_ju_\lambda\|_{L^qL^2(S_{j,k})}  \lesssim \lambda^{\gamma}\theta_j^\sigma \left(\|u_\lambda\|_{L^\infty L^2(S)} +\|(D_{x_1}-Q_\lambda(x,D_{x'}))u_\lambda\|_{L^1L^2(S)}\right).
\end{equation}
where we take the norms on the left and right to mean
\begin{align*}
L^qL^2(S_{j,k}) &= L^q([k\tilde{\veps}\theta_j,(k+1)\tilde{\veps}\theta_j]_{x_1}\times\RR^{n-1}_{x_2,\dots,x_n};L^2(\RR_{x_0})),\\L^pL^2(S)&=L^p([-1,1]_{x_1};L^2(\RR^{n}_{x'})), \qquad p=1,\infty.
\end{align*}
Indeed, this is a consequence of Theorem 3.1 (when $n=2$) and the discussion preceding Theorem 7.2 (when $n\geq 3$) in \cite{smithsogge07}, along with the flux estimates in \S6 there.\footnote{Note that the notation varies slightly in this work as we are taking a base $\sqrt{2}$ decomposition in $|\xi_n|/\lambda$ instead of the usual dyadic decomposition, in particular \cite{smithsogge07} denotes $\theta= 2^{-j}$. Also, strictly speaking in \cite{smithsogge07}, $L^1L^2(S)$ is replaced by $L^2(S)$, but given Duhamel's principle, the former is acceptable.}  We stress that this result holds for any $2^{-\frac j2} \in [\lambda^{-\frac 13}, 1]$, not just ones which satisfy $j\geq J_\alpha$ as above.  Moreover, as indicated there, the same result holds if the microlocal cutoff $\Gamma_j$ truncates to $|\xi_n|\lesssim \lambda^{2/3}$ instead.  This is equivalent to saying that $|\xi_n| \lesssim \lambda\theta$ when $\theta = \lambda^{-1/3}$, so we refer to this as the ``$\theta\approx \lambda^{-1/3}$" case below.

We will provide a brief sketch of \eqref{smsosqfcn} below for the convenience of the reader and to show the flexibility of the argument; however, we stress that what will appear is merely a summary, and the estimate is due to Smith and Sogge.  For now, we observe that it yields the bounds \eqref{wjbound}.  To this end, note that if $\phi_{j,k}(x_1)$ is a bump function supported in $[(k-1)\tilde{\veps}\theta_j,(k+2)\tilde{\veps}\theta_j] $ which is identically one for $x_1 \in [k\tilde{\veps}\theta_j,(k+1)\tilde{\veps}\theta_j] $, then the Duhamel formula allows us to see that
\begin{multline}\label{doubletrick}
 \|\phi_{j,k} w_j \|_{L^\infty L^2(S)} + \|(D_{x_1}-Q_\lambda(x,D_{x'}))(\phi_{j,k} w_j)\|_{L^1L^2(S)}\\
 \lesssim
2^{\frac j4}\|w_j\|_{L^2(S_{j,k}^*)} + 2^{-\frac j4}\|(D_{x_1}-Q_\lambda(x,D_{x'})) w_j\|_{L^2(S_{j,k}^*)}
\end{multline}
where $S_{j,k}^*= [(k-1)\tilde{\veps}\theta_j,(k+2)\tilde{\veps}\theta_j] \times \RR^{n}_{x'}$.\footnote{The cutoffs $\Gamma_j$ are denoted as $\beta_j$ in \cite{smithsogge07}.  Given the bound \eqref{psicomm}, it is convenient to include the multiplier $\Gamma_j$ in the driving force instead of estimating its contribution as in \S6.4 of that work.}  Thus \eqref{smsosqfcn} implies that $\|w_j\|_{L^qL^2(S_{j,k})} $ is dominated by this quantity. Since $S_{j,k}^*\cap S_{j,k'}^*\neq \emptyset$ implies that $|k-k'|\leq 1$, we obtain \eqref{wjbound} by taking a sum in $k$.

\begin{proof}[Proof sketch of \eqref{smsosqfcn}]
Let $Q_j$ be the operator obtained by replacing $Q_\lambda$ by the operator obtained by truncating the symbol $q_\lambda$ to frequencies less than $\lambda^{1/2}\theta_j^{-1/2} = \lambda^{1/2}2^{j/4}$.  Suppose
\[
(D_{x_1}-Q_j(x,D_{x'}))u_\lambda = F_j + G_j
\]
when $\theta_j > \lambda^{-\frac 13}$, we claim that
\begin{multline}\label{smsoprep}
\|\Gamma_ju_\lambda\|_{L^qL^2(S_{j,k})}  \lesssim \lambda^{\gamma}\theta_j^{\sigma}\Big( \|\Gamma_ju_\lambda\|_{L^\infty L^2(S_{j,k})} +\lambda^{\frac 14}\theta_j^{\frac 14}\|\langle\lambda^{\frac 12}\theta_j^{-\frac 12}x_n\rangle^{-1} \Gamma_ju_\lambda\|_{L^2(S_{j,k})}\\
+ \|F_j\|_{L^1L^2(S_{j,k})} + \lambda^{-\frac 14}\theta_j^{-\frac 14}\|\langle\lambda^{\frac 12}\theta_j^{-\frac 12}x_n\rangle^2G_j\|_{L^2(S_{j,k})}\Big),
\end{multline}
and when $\theta_j \approx \lambda^{-\frac 13}$ (recall that this is the case of a cutoff to $\{|\xi_n| \lesssim \lambda^{2/3}\}$), we claim that
\begin{equation}\label{smsoprep2}
\|\Gamma_ju_\lambda\|_{L^qL^2(S_{j,k})}  \lesssim \lambda^{\gamma}\theta_j^{\sigma}( \|\Gamma_ju_\lambda\|_{L^\infty L^2(S_{j,k})} + \|F_j+G_j\|_{L^1L^2(S_{j,k})}),
\end{equation}
The aforementioned flux estimates in \cite[\S6]{smithsogge07} show that the right hand side here is in turn uniformly bounded by the right hand side of \eqref{smsosqfcn} (cf. \cite[(3.1)]{smithsogge07}).  In particular, we have for $|\xi'|\approx 1$ (which is \cite[(6.31)]{smithsogge07} scaled back using $x\mapsto \theta_j^{-1}x$)
\begin{equation}\label{qjapprox}
|q_j(x,\xi')-q_\lambda(x,\xi')| \lesssim_N \lambda^{-\frac 12} \theta_j^{\frac 12}\langle \lambda^{\frac 12} \theta_j^{-\frac 12} x_n\rangle^{-N}
\end{equation}
and hence the error induced by replacing $Q_\lambda$ by $Q_j$ can be absorbed into $G_j$.  The key idea is that the singular contribution of $\prtl^2_{x_n} (g^{lm}(x',|x_n|))$ is $2\prtl_{x_n}g^{lm}(x',0)\delta(x_n)$, hence regularizing the symbol in this manner results in such tails.

To see \eqref{smsoprep}, we may translate to $k=0$ and dilate by a factor of $\theta_j$, thus considering (with a slight abuse of notation) $u_\mu(x)= (\Gamma_ju_\lambda)(\theta_j x)$, which is now localized at a frequency scale $|\xi|\approx\mu: = \lambda \theta_j$ with $|\xi_n| \approx \mu\theta_j$ when $\theta_j >\lambda^{-\frac 13}$ or $|\xi_n| \lesssim \mu^{\frac 12}$ when $\theta_j \approx \lambda^{-1/3}$.  Note that the slab $S_{j,0}$ dilates to a unit slab $S=[0,\tilde{\veps}] \times \RR^n$ in the new coordinates.  We write $(D_{x_1}-Q_\mu(x,D_{x'}))u_\mu = F + G$ where the symbol of $Q_\mu$ is $q_\mu(x,\xi') = \theta_j q_j (\theta_j x, \theta_j^{-1}\xi')$.  When $\theta_j > \mu^{-\frac 12}$ (equivalently $\theta_j > \lambda^{-1/3}$ in the old coordinates), the bound \eqref{smsoprep} rescales to
\begin{multline*}
\|u_\mu\|_{L^qL^2(S)}  \lesssim \mu^{\gamma}\theta_j^{\sigma}\Big( \|u_\mu\|_{L^\infty L^2(S)}
+\mu^{\frac 14}\theta_j^{\frac 12}\|\langle\mu^{\frac 12}x_n\rangle^{-1} u_\mu\|_{L^2(S)}\\
+ \|F\|_{L^1L^2(S)} + \mu^{-\frac 14}\theta_j^{-\frac 12}\|\langle\mu^{\frac 12}x_n\rangle^2G\|_{L^2(S)}\Big),
\end{multline*}
and when $\theta_j \approx \mu^{-\frac 12}$ (equivalently $\theta_j \approx \lambda^{-1/3}$), the bound \eqref{smsoprep2} rescales to
\begin{equation*}
\|u_\mu\|_{L^qL^2(S)}  \lesssim \mu^{\gamma}\theta_j^{\sigma}( \|u_\mu\|_{L^\infty L^2(S)} + \|F+G\|_{L^1L^2(S)}),
\end{equation*}

We now suppress the dependence of $\theta$ on $j$ for the remainder of the argument.  Consider the wave packet transform of a function $f(y')$, $y'\in \RR^n$,
\[
T_\mu f(x',\xi') = \mu^{\frac n4} \int e^{-i\langle \xi',y'-x'\rangle} g(\mu^{\frac 12}(y'-x'))f(y')\,dy',
\]
where $\widehat{g}$ is smooth, radial, and compactly supported in a small ball about the origin with $\|g\|_{L^2(\RR^n)} = (2\pi)^{-n/2}$.  The transformation satisfies $T^*_\mu T_\mu=I$ and hence $T_\mu: L^2(\RR^n_{y'}) \to L^2(\RR^{2n}_{x',\xi'})$ is an isometry.  Let $H_{q_\mu} = d_{x'}q_\mu\cdot d_{\xi'} - d_{\xi'} q_\mu \cdot d_{x'} $ denote the Hamiltonian vector field of $q_\mu$.  When $\theta \approx \mu^{-\frac 12}$, 
\[
|\prtl^{\beta_1}_{x'}\prtl^{\beta_2}_{\xi'}q_\mu(x,\xi')| \lesssim_{\beta_1,\beta_2} \mu^{\frac 12\max(0,|\beta_1|-2)} \qquad \text{ for } |\xi'|\approx 1,
\]
that is, $q_\mu(\cdot,\xi')$ behaves as a $C^2$ symbol truncated to frequencies less than $\mu^{1/2}$.  This is the threshold by which it can be seen that $T_\mu Q-iH_{q_\mu}T_\mu: L^2(\RR^n_{y'}) \to L^2(\RR^{2n}_{x',\xi'})$ is uniformly bounded.  Hence $\tilde{u}(x,\xi') := T_\mu(u(x_1,\cdot))(x',\xi')$ solves a transport equation in $x_1,x',\xi'$ with bounded driving force.  When $\theta > \mu^{-\frac 12}$, $q_\mu$ instead satisfies bounds
\[
|\prtl^{\beta_1}_{x'}\prtl^{\beta_2}_{\xi'}q_\mu(x,\xi')| \lesssim_{\beta_1,\beta_2,N}
1+\mu^{\frac 12(|\beta_1|-1)}\theta\langle\mu^{\frac 12} x_n\rangle^{-N}
\]
for $|\xi'|\approx 1$ (cf. \cite[(6.32)]{smithsogge07}).  Therefore by \cite[Lemma 4.4]{smithsogge07}, $\tilde{u}$ instead solves the transport equation
\[
(\prtl_{x_1}-H_{q_\mu})\tilde{u}(x,\xi') = \tilde{F}(x,\xi') + \tilde{G}(x,\xi')
\]
where over $\tilde{S} =[0,\tilde{\veps}] \times \RR^{2n}_{x',\xi'}$,
\begin{multline}\label{rescaledweighted}
\| \tilde{F}\|_{L^1L^2(\tilde{S})} + \mu^{-\frac 14}\theta^{-\frac 12} \|\langle\mu^{\frac 12} x_n\rangle^2\tilde{G}\|_{L^2(\tilde{S})} \lesssim \|u_\mu\|_{L^\infty L^2} + \mu^{\frac 14}\theta^{\frac 12} \|\langle\mu^{\frac 12} x_n\rangle^{-1}u_\mu \|_{L^2(S)}\\
+\| F\|_{L^1L^2(S)} + \mu^{-\frac 14}\theta^{-\frac 12} \|\langle\mu^{\frac 12} x_n\rangle^2G\|_{L^2(S)}.
\end{multline}
The main idea is that since packets are spatially concentrated within a distance $\mu^{-1/2}$, the conjugation error $T_\mu Q-iH_qT_\mu$ can be bounded by employing the weighted $L^2$ estimates.  In particular, the proof uses that if $\supp(\widehat{f}) \subset \{|\xi_n|\approx \mu\}$,
\begin{equation}\label{conjugerror}
\|(T_\mu Q-iH_qT_\mu)f\|_{L^2(\RR^{2n}_{x',\xi'})}\lesssim_N \|f\|_{L^2} + \mu^{\frac 12}\theta\|\langle\mu^{\frac 12}x_n\rangle^{-N} f\|_{L^2},
\end{equation}
and hence the exact powers of $\langle\mu^{\frac 12} x_n\rangle$ in \eqref{rescaledweighted} are not crucial.  By the same idea, we have that $\|\langle \mu^{\frac 12}x_n\rangle^{-N}T_\mu u_\mu\|_{L^2(\tilde{S})} \lesssim \|\langle \mu^{\frac 12}x_n\rangle^{-N}u_\mu\|_{L^2(S)} $ (cf. \cite[(4.3)]{smithsogge07}.

Let $\Theta_{0,r}(x',\xi')$ denote the integral curve of $H_{q_\mu}$ satisfying $\Theta_{0,r}(x',\xi')|_{r=0}=(x',\xi')$.  It is shown that by employing Duhamel's principle and the $V^2_q$ spaces of Koch and Tataru, the desired $L^qL^2$ bounds on $u_\lambda $ follow from the following estimate on functions $\tilde{f}(x',\xi')\in L^2(\RR^{2n})$ such that $\supp(\tilde{f}) \subset \{|\xi_n|\approx \mu\theta \}$
\begin{equation}\label{Wdef}
 \|W\tilde{f}\|_{L^qL^2(S)} \lesssim \mu^\gamma \theta^\sigma \|\tilde{f}\|_{L^2(\RR^{2n}_{x',\xi'})}, \qquad W\tilde{f}(x) = T_\mu^*(\tilde{f}\circ \Theta_{0,x_1})(x').
\end{equation}
Indeed, the compact support of $\widehat{g}$ in the definition of $T_\mu$ means that $\supp(\tilde{u}) \subset \{|\xi_n|\approx \mu\theta\}$. The kernel $K(x,y)$ of $WW^*$ can be written (cf. \cite[p.131]{smithsogge07})
\begin{equation*}
\mu^{\frac n2}\int e^{i\langle \zeta,x'-z \rangle-i\langle \zeta_{y_1,x_1},y'-z_{y_1,x_1}\rangle } g(\mu^{\frac 12}(x'-z)) g(\mu^{\frac 12}(y'-z_{y_1,x_1}))\tilde{\Gamma}_\theta(\zeta)dzd\zeta
\end{equation*}
where $\tilde{\Gamma}_\theta$ is supported in $\{|\xi_n|\approx \mu\theta\}$.  Consequently, the bound on $W$ follows from showing that the kernel $K(x,y)$ satisfies (cf. \cite[p.150]{smithsogge07})
\begin{equation}\label{wpkernelbound}
\int|K(x,y)|\,dx_0 + \int|K(x,y)|\,dy_0 \lesssim \mu^{n-1}\theta (1+\mu|x_1-y_1|)^{-\frac{n-2}{2}}(1+\mu\theta^2|x_1-y_1|)^{-\frac 12} .
\end{equation}
Indeed, this yields an $L^1_{x_2,\dots,x_n}L^2_{x_0}\to L^\infty_{x_2,\dots,x_n}L^2_{x_0}$ bound which can be interpolated with trivial $L^2$ bounds to obtain the following for fixed $x_1,y_1$
\[
\left\| \int K(x,y)F(y)\,dy' \right\|_{L^qL^2(\RR^n_{x'})} \lesssim \mu^{2\gamma}\theta^{2\sigma}|x_1-y_1|^{-\frac 2q}\|F(\cdot,y_1,\cdot)\|_{L^{q'}L^2(\RR^n_{x'})}
\]
This is because we may assume that $q$ is sufficiently close to, but strictly greater than, $\frac{2(n+1)}{n-1}$, we may sacrifice as much of $(1+\mu\theta^2|x_1-y_1|)^{-\frac 12}$ as is needed to obtain the decay in $|x_1-y_1|$, the rest contributes to a gain in $\theta$.  At this point, the Hardy-Littlewood-Sobolev inequality shows that $\|WW^*F\|_{L^qL^2(S)} \lesssim \mu^{2\gamma}\theta^{2\sigma}\|F\|_{L^{q'}L^2(S)}$.

The main idea behind \eqref{wpkernelbound} is that $K$ is concentrated in a $\mu^{-1}$ neighborhood of the light cone, so integration in $y_0$ always yields a gain on that scale.  The gains in $\theta$ can be motivated by considering the three cases $|x_1-y_1| \leq  \mu^{-1}$, $\mu^{-1} < |x_1-y_1| \leq \mu^{-1}\theta^{-2} $, and $\mu^{-1}\theta^{-2} < |x_1-y_1| $.  In the first case, the separation of $x_1,y_1$ generates negligible oscillations, so one simply picks up the volume of the support of $K$.  In the last case, the separation in  $x_1,y_1$ generates the usual decay in $|x_1-y_1|$.  The second case is thus intermediate to these two extremes, one obtains the best estimate by exploiting oscillations in all variables except for $\xi_n$, thus obtaining adjusted decay in $|x_1-y_1|$ while gaining $\mu\theta$, which is the volume of the $\xi_n$ projection of $\supp(\Gamma_\theta)$.
\end{proof}

We now turn to the proof of \eqref{vjbound}.  Given \eqref{qjapprox}, we may write $F_j=(D_{x_1}-Q_j(x,D))v_j$ so that it suffices to show the following analogue of \eqref{smsoprep}
\begin{multline*}
\|v_j\|_{L^qL^2(S_{j,k})}  \lesssim \\\lambda^{\gamma}2^{-\frac{j\sigma}{2}}\Big( \|v_j\|_{L^\infty L^2(S_{j,k})}
+\lambda^{\frac 14}2^{-\frac j4}\|\langle\lambda^{\frac 12}2^{\frac j4}x_n\rangle^{-M} v_j\|_{L^2(S_{j,k})}
+ \|F_j\|_{L^1L^2(S_{j,k})} \Big),
\end{multline*}
as the error $(Q_\lambda-Q_j)v_j$ can be absorbed by the middle term in \eqref{vjbound}. Also, commuting the equation with a cutoff $\phi_{j,k}$ as before, the $L^pL^2(S_{j,k})$ spaces can be replaced with weighted $L^2(S_{j,k}^*)$ spaces.  Next we rescale the problem by $2^{-\frac j2}$, setting, $\mu = \lambda 2^{-\frac j2}$ and $v_\mu(x)=v_j(2^{-\frac j2}x)$, $F_\mu(x)=2^{-\frac j2}F_j(2^{-\frac j2}x)$.  We then use the wave packet transform as before, setting $\tilde{v} = T_\mu v_\mu$ , $(\prtl_{x_1}-H_{q_\mu})\tilde{v}= \tilde{F}$. Using \eqref{conjugerror} with $\theta=2^{-\frac j2}$, we have that
\begin{equation*}
\| \tilde{F}\|_{L^1L^2(\tilde{S})} \lesssim_N \|u\|_{L^\infty L^2(S)} + \mu^{\frac 14}2^{-\frac j4} \|\langle\mu^{\frac 12} x_n\rangle^{-M} v \|_{L^2(S)}
+\| F_\mu\|_{L^1L^2(S)}.
\end{equation*}
We are thus reduced to showing \eqref{Wdef} with $\theta = 2^{-\frac j2}$ and $\supp(\tilde{f})$ contained in a set of the form $\{|\xi_n|\lesssim \mu\theta\}$.  While in previous works, \eqref{Wdef} is shown under the assumption that the support of $\tilde{f}$ is instead of the form $\{|\xi_n|\approx \mu\theta\}$, tracing through the steps of the proof verifies that the larger support presents no additional complication. Alternatively, one can simply use a smooth partition of unity
\[
\Gamma_j(\xi_n) = \Gamma_{<\lfloor \log_2 (\lambda^{2/3}) \rfloor}(\xi_n) + \sum_{j \leq l \leq \lfloor \log_2 (\lambda^{2/3}) \rfloor }\Gamma_j(\xi_n),
\]
partitioning $\supp(\tilde{f})$ into cones of smaller angles so that \eqref{Wdef} applies to each term.

The bound \eqref{vJbound} follows by similar considerations, but this time we truncate $q$ to frequencies less than $\lambda^{2/3}$ and take the dilation $x \mapsto \lambda^{-1/3} x$ so that $\mu = \lambda^{2/3}$ and that in the new coordinates, $|\prtl^{\beta_1}_x\prtl^{\beta_2}_{\xi'} q_\mu(x,\xi')| \lesssim \mu^{\max(0,|\beta_1|-2)}$ for $|\xi'|\approx 1$.  Therefore as observed above, conjugating $Q_\mu$ by the wave packet transform introduces bounded error.   Duhamel's principle means that we are reduced to showing that with $W$ defined as before, $\|W\tilde{f}\|_{L^qL^2(S)} \lesssim \mu^\gamma 2^{-J_\alpha/2} \|\tilde{f}\|_{L^2}$ when $\supp(\tilde{f}) \subset \{|\xi_n|\lesssim \mu 2^{-J_\alpha/2} \}$.  The bound thus follows by the same considerations as in the $v_j$ case.

\subsection{Strichartz estimates}
When $X=L^p L^q$ we instead factorize the principal symbol of $P$ as a quadratic in $\xi_0$ instead of $\xi_1$.
\[
g^{00}(x)\left( \xi_0 + q^+_\lambda(x,\xi') \right)\left( \xi_0 - q^-_\lambda(x,\xi') \right) \quad \xi'=(\xi_1,\xi_2,\dots,\xi_n),\; x'=(x_1,x_2,\dots,x_n),
\]
again with $q^\pm >0$ on the support of $\widehat{u}_\lambda$.  Working with the half wave operator $D_{x_0}-Q_\lambda(x,\xi')$, the proofs of \eqref{vjbound}, \eqref{vJbound}, \eqref{wjbound} all follow by the same procedure as before.  Indeed, this is the key observation in \cite{bssbdrystz}.  The only crucial difference is that the integration in \eqref{wpkernelbound} is not needed, one simply observes that
\begin{equation}\label{wpkernelboundstz}
|K(x,y)| \lesssim \mu^n \theta(1+\mu|x_0-y_0|)^{-\frac{n-2}{2}}(1+\mu\theta^2|x_0-y_0|)^{-\frac 12}
\end{equation}
so that for $\frac{n-1}2-\frac{n-1}q -\frac 2p$ sufficiently small, interpolation gives
\[
\left\| \int K(x_0,\cdot,y_0,y')F(y_0,y')\,dy' \right\|_{L^q(\RR^n_{x'})} \lesssim \mu^{2\gamma}\theta^{2\sigma}|x_0-y_0|^{-\frac 2p} \|F(\cdot,y_0,\cdot)\|_{L^{q'}(\RR^n_{x'})}
\]
with $x'=(x_1,\dots,x_n)$ again by sacrificing only as much of the last factor on the right in \eqref{wpkernelboundstz} as needed to obtain the $|x_0-y_0|^{-\frac 2p} $ decay.

\section{Intrinsic localized energy estimates}\label{sec:intrinsic}
Here we prove Theorem \ref{thm:intrinsic}, then later verify the first part of Theorem \ref{thm:blairrls}.  Since the estimates here are for time-independent metrics, we use both $t$ and $x_0$ to denote the time coordinate.  By Duhamel's formula, it suffices to assume that $(D_t^2-\Delta_g)u_\mu=0$.  By taking a finite partition of unity it suffices to prove estimates on $\phi u_\mu$, where $\phi$ is a smooth bump function supported in a suitable coordinate system.  Specifically, we use the coordinate system outlined in \S\ref{ss:coordintro}, recalling that this can achieved by a transformation which is independent of $t$.  We assume
that in these coordinates, $\phi$ is supported in $\{|(t,x)|_\infty\leq 2\}$, identically one on $\{|(t,x)|_\infty\leq 1\}$ and that $\phi$ is independent of $x_n$ near $x_n=0$.  We may also suppose that the metric is extended to be flat for $|x|_\infty \geq 3$.

Recall that with $\rho(x) = \sqrt{\det g_{lm}(x)}$, we have that in our coordinate system,
\[
\Delta_g f = \frac{1}{\rho(x)} \sum_{i,j=1}^{n}D_{x_i}\left(g^{ij}(x)\rho(x)D_{x_j}f\right), \qquad g^{in}(x)=\delta_{in}.
\]
As in \S\ref{ss:coordintro}, we take an odd or even extension of $\phi u_\mu$ across $x_n=0$ (for Dirichlet or Neumann conditions respectively) and a corresponding odd or even extension of $(D_t^2-\Delta_g)(\phi u_\mu)$ across $x_n=0$.  Therefore in what follows, $\Delta_g$ denotes the differential operator obtained by extending the coefficients of $\Delta_g$ evenly across $x_n=0$.  We also abbreviate $P=D_t^2-\Delta_g$.  The desired estimates will follow from a further frequency decomposition determined by the Fourier transform in the coordinate system.  Since this requires us to examine frequency scales with respect to the Fourier transform which are less than $\mu$, we introduce $\tilde{u}_\mu := (\mu^{-2}\Delta_g)^{-1}u_\mu$ which by the functional calculus, satisfies
\[
\|\tilde{u}_\mu(t,\cdot)\|_{L^2( M)}\approx\|u_\mu(t,\cdot)\|_{L^2( M)}.
\]

Let $\psi_j$ be a smooth bump function identically one on $\{|x_n|\leq 2^{-j}\}$ and supported in $\{|x_n|\leq 2^{-j+1}\}$.
Now define the seminorm
\[
 \|u\|_{E_j}:=
2^{\frac j4}\sum_{i=0}^{n-1}\|\psi_j D_{x_i}u\|_{L^2(\RR^{n+1})} + 2^{\frac j2}\|\psi_j D_{x_n}u\|_{L^2(\RR^{n+1})},
\]
so that it is sufficient to show that for an implicit constant independent of $\mu$, $j$,
\begin{multline}\label{ejbound}
( \mu2^{\frac j4})^2\|\psi_j\phi u_\mu\|_{L^2(\RR^{n+1})}^2 +\|\phi u_\mu\|_{E_j}^2 \lesssim\\
\mu\|\tilde{\phi}u_\mu\|_{L^2(\RR^{n+1})}^2+ \|D(\tilde{\phi} u_\mu)\|_{L^2(\RR^{n+1})}^2 + \|P(\phi u_\mu)\|_{L^2(\RR^{n+1})}^2+\\
\mu\|\tilde{\phi}\tilde{u}_\mu\|_{L^2(\RR^{n+1})}^2+ \|D(\tilde{\phi} \tilde{u}_\mu)\|_{L^2(\RR^{n+1})}^2 + \|P(\phi \tilde{u}_\mu)\|_{L^2(\RR^{n+1})}^2  + \frac{1}{\mu}\sum_{i=1}^n\|\tilde{\phi}D_{x_i} \prtl_t \tilde{u}_\mu\|_{L^2(\RR^{n+1})}^2
\end{multline}
where $\tilde{\phi}$ denotes a vector of bump functions with support contained in $\supp(\phi)$ and $D=(D_{x_0},D_{x_1},\dots,D_{x_n})$.  Indeed, if such estimates are valid, then energy estimates show that the right hand side is bounded by the right hand side of \eqref{intrinsiclocnrg}.

We use the decomposition \eqref{littpaleylambda} from above with respect to the Fourier transform in our coordinates, but replacing $\Lambda u$ by $\phi u_\mu$.  We first note that the family $\{\psi_j \beta_l^N \beta_k^T (\phi u_\mu)\}_{l>k}$ is almost orthogonal in frequency, as $\widehat{\psi_j}$ rapidly decreasing outside of $|\xi_n| \lesssim 2^j\leq\mu^{\frac 23}$.  Consequently, we have the following bound on the terms where the normal frequencies dominate both $\mu$ and the tangential frequencies
\begin{multline*}
(\mu2^{\frac j4})^2\left\|\psi_j\sum_{k=1}^\infty \sum_{l>\max(k,\log_2\mu)} \beta^N_{l}\beta^T_k(\phi u_\mu)\right\|_{L^2}^2 +
 \left\|\sum_{k=1}^\infty \sum_{l>\max(k,\log_2\mu)} \beta^N_{l}\beta^T_k(\phi u_\mu)\right\|_{E_j}^2\\
\lesssim \sum_{k=1}^\infty \sum_{l>\max(k,\log_2\mu)} \left( (\mu2^{\frac j4})^2\left\| \beta^N_{l}\beta^T_k(\phi u_\mu)\right\|_{L^2}^2+\left\| \beta^N_{l}\beta^T_k(\phi u_\mu)\right\|_{E_j}^2 \right).
\end{multline*}
Next we have that
\begin{multline}\label{normalest}
(\mu2^{\frac j4})^2\left\|\psi_j \beta^N_{l}\beta^T_k(\phi u_\mu)\right\|_{L^2}^2+ \left\| \beta^N_{l}\beta^T_k(\phi u_\mu)\right\|_{E_j}^2\\ \lesssim \mu^2\|\beta^N_{l}\beta^T_k(\phi u_\mu)\|_{L^2(\RR^{n+1})}^2 + \|D(\beta^N_{l}\beta^T_k(\phi u_\mu)\|_{L^2(\RR^{n+1})}^2+
\|P(\beta^N_{l}\beta^T_k(\phi u_\mu))\|_{L^2(\RR^{n+1})}^2.
\end{multline}
In fact, a stronger estimate holds as the gain $2^{-\frac j4}$ can be replaced by $2^{-\frac j2}$ in all cases.  This can be seen either by using the angle one parametrix in \cite{smithsogge07} surveyed in \S\ref{ss:sqfcncoord} or by using that the solution can be represented as a sum of Fourier integral operators of the proper order.  Indeed, the Fourier localization to a cone $\{ |\xi_n| \gtrsim |\xi'|\}$ means that the solution is concentrated along rays which reflect in the boundary at an angle uniformly bounded from below and hence the gain in regularity is determined by the fact that packets will escape the region $|x_n| \lesssim 2^{-j}$ in a time scale comparable to $2^{-j}$.

We now sum over the terms on the right hand side of \eqref{normalest} over $l>k$.   The terms not involving $P$ present no problem as the sum over these terms are bounded by the first term in \eqref{ejbound}.  We next observe that for an arbitrary sequence $\veps_l$ taking on values $\pm 1$ and any Lipschitz function $a$, the Coifman-Meyer commutator theorem shows that $[a,\sum_l \veps_l \beta^N_l]$ maps $L^2 \to H^1$.  Hence as in \eqref{khinchin}, we have
\[
\sum_{k=1}^\infty\sum_{l=k+1}^\infty \|P(\beta^N_{l}\beta^T_k(\phi u_\mu))\|_{L^2}^2 \lesssim \sum_{k=1}^\infty\|\beta^T_k(\phi u_\mu)\|_{H^1(\RR^{n+1})}^2 +
\sum_{k=1}^\infty\|P(\beta^T_k(\phi u_\mu))\|_{L^2}^2 .
\]
An easier commutator argument then shows that the sum on the right is bounded by the right hand side of \eqref{normalest}.

Theorem \ref{thm:maincoordestimate} and a similar commutator argument shows that
\[
 (\mu2^{\frac j4})^2\left\|\psi_j\sum_{k>\log_2\mu} \beta^N_{<k}\beta^T_k(\phi u_\mu)\right\|_{L^2}^2 +
 \left\|\sum_{k>\log_2\mu} \beta^N_{<k}\beta^T_k(\phi u_\mu)\right\|_{E_j}^2
\]
is also bounded by the right hand side of \eqref{ejbound}.

We are now left to handle the cases where the frequency scale given by coordinates are less than $\mu$, which requires us to consider $\tilde{u}_\mu$ as defined above, which satsifies $u_\mu = \mu^{-2}\Delta_g \tilde{u}_\mu$.  First observe that for $|\alpha|\leq 1$,
\begin{equation}\label{doublesum}
 \left\| \sum_{k=1}^{\log_2\mu} \sum_{l=k}^{\log_2\mu}\psi_jD^\alpha \beta_l^N\beta_k^T (\phi u_\mu) \right\|_{L^2}^2 \lesssim \sum_{k=1}^{\log_2\mu}  \left(\sum_{l=k}^{\log_2\mu}\left\| \psi_jD^\alpha \beta_l^N\beta_k^T (\phi u_\mu) \right\|_{L^2}\right)^2 ,
\end{equation}
since we can use almost orthogonality in $k$, but not in $l$.  We will make a slight abuse of notation here and below, treating the $l=k$ term in the second sum as the result of replacing $\beta_l^N$ by $\beta_{<k}^N$.  Next observe that
\begin{equation}\label{massbump}
\mu2^{\frac j4}\left\|\psi_j\beta^N_{l}\beta^T_k(\phi u_\mu)\right\|_{L^2} \lesssim \mu2^{\frac j4-k}\left\|\psi_j\beta^N_{l}\beta^T_k D_T(\phi u_\mu)\right\|_{L^2}
\end{equation}
where in slight contrast to \S\ref{ss:poscomm}, we abbreviate $D_T = (D_{x_1},\dots,D_{x_{n-1}})$ without the time derivative.  This allows us to lump the mass terms in our estimates in with the derivatives over $i=1,\dots,n-1$.

We have that for $i=0,\dots, n$
\begin{multline}\label{intermednrml}
D_{x_i}\beta_l^N\beta_k^T (\phi u_\mu) - \mu^{-2}\Delta_g\left( D_{x_i}\beta_l^N\beta_k^T (\phi \tilde{u}_\mu) \right) = \\
\mu^{-2}D_{x_i}\beta_l^N\beta_k^T ([\phi,\Delta_g] \tilde{u}_\mu) + \mu^{-2}[D_{x_i}\beta_l^N\beta_k^T,\Delta_g]\phi \tilde{u}_\mu.
\end{multline}
The two terms on the right are of lower order, in that we have the following bound which does not use any restiction to $S_{<j}$
\begin{equation}\label{intermednrmlbound}
 \|\eqref{intermednrml}\|_{L^2(\RR^{n+1})} \lesssim \mu^{-2}c_{i,k,l}\sum_{m=1}^n\|\tilde{\phi}D_{x_m}u_\mu\|_{L^2(\RR^{n+1})} + \mu^{-2}\|\tilde{\phi}D_T\prtl_t u_\mu\|_{L^2(\RR^{n+1})}
\end{equation}
where $c_{i,k,l}=2^k$ when $i=0,\dots,n-1$ and $c_{i,k,l}=2^l$ when $i=n$. Note that when $i=1,\dots,n-1$, the loss of only $2^k$ counterbalances the smaller gain presented by \eqref{massbump}. The very last term involving $D_T\prtl_t u_\mu$ only comes into play when $i=0$.  Proving such estimates on the first term in \eqref{intermednrml} is thus straightforward.  To see how to bound the second term in \eqref{intermednrml}, note that the highest order contribution of $[D_{x_i}\beta_l^N\beta_k^T,\Delta_g]\phi \tilde{u}_\mu$ can be estimated by writing
\[
[\tilde{\beta}_l^N,g^{ij}]\tilde{\beta}_k^TD_{x_i}D_{x_j}(\phi \tilde{u}_\mu)+\tilde{\beta}_l^N[\tilde{\beta}_k^T,g^{ij}] D_{x_i}D_{x_j}(\phi \tilde{u}_\mu)
\]
where $\tilde{\beta}_l^N$ is either of the form $\beta_l^N$ or $2^{-l}D_{x_n}\beta_l^N$ and $\tilde{\beta}_k^T$ is either of the form $\beta_k^T$ or $2^{-k}D_{x_i}\beta_k^T$, $i=1,\dots,n-1$.  Since we only need to consider tangential derivatives $D_{x_i}D_{x_j}$, $i,j=1,\dots,n-1$, in this expression, the first term yields a net gain of $2^{k-l}$ and the second can be bounded using that the smoothness of $g^{ij}$ in $x_1,\dots,x_{n-1}$ means $[\tilde{\beta}_k^T,g^{ij}] D_{x_i}$ is bounded on $L^2$.  The lower order contributions of $[D_{x_i}\beta_l^N\beta_k^T,\Delta_g]\phi \tilde{u}_\mu$ can be handled just by using that the commutator of $\tilde{\beta}_l^N,\tilde{\beta}_k^T$ with an $L^\infty$ function is merely bounded on $L^2$.

Given \eqref{intermednrmlbound}, the $L^2(\RR^{n+1})$ norm of \eqref{intermednrml} exhibits an overall gain of $\mu^{-1}$, which is stronger than the gains of $2^{j/2}$ even after logarithmic losses from summation in $k,l$.  Note that the same principle works when the $\beta_l^N$ is replaced by a $\beta_{<k}^N$.

Now observe that
\begin{multline*}
\mu^{-1}2^{\frac j4-k}\sum_{i=0}^{n-1} \|\psi_j\Delta_g\left( D_{x_i}\beta_l^N\beta_k^T (\phi \tilde{u}_\mu) \right)\|_{L^2} +
 \mu^{-2}2^{\frac j2} \|\psi_j\Delta_g\left( D_{x_n}\beta_l^N\beta_k^T (\phi \tilde{u}_\mu) \right)\|_{L^2}\\
\lesssim \mu^{-1}2^{\frac j4+l}\sum_{i=0}^{n-1} \|\psi_j D_{x_i}\tilde{\beta}_l^N\tilde{\beta}_k^T (\phi \tilde{u}_\mu) \|_{L^2}
+\mu^{-2}2^{\frac j2+2l} \|\psi_j D_{x_n}\tilde{\beta}_l^N\tilde{\beta}_k^T (\phi \tilde{u}_\mu) \|_{L^2}
\end{multline*}
where $\tilde{\beta}_l^N$, $\tilde{\beta}_k^T$ denote vectors of Fourier multipliers with symbols supported in $\supp(\beta_l^N)$, $\supp(\beta_k^T )$ respectively.  We now claim that this in turn is bounded by
\[
\left( \frac{2^{l+\max(0,\frac j4-\frac l6)}}{\mu}+\frac{2^{2l+\max(0,\frac j2-\frac l3)}}{\mu^{2}}\right)
\left(\|\tilde{\beta}_l^N\tilde{\beta}_k^T (\phi \tilde{u}_\mu) \|_{H^1} +
 \big\|P\big(\tilde{\beta}_l^N\tilde{\beta}_k^T (\phi \tilde{u}_\mu)\big) \big\|_{L^2} \right).
\]
But this is a consequence of either using the parametrix as in \eqref{normalest} or applying Theorem \ref{thm:maincoordestimate} as before in the cases $l>k$ and $l=k$ respectively.  Indeed, when  $2^{-j}\geq 2^{-2l/3}$, the width of the collar is large enough relative to the frequency scale, so that we obtain the usual gain.  Otherwise if $2^{-j}< 2^{-2l/3}$ we obtain estimates by restricting to a larger collar of width $2^{-l}$.
Consequently, we may split the right hand side of \eqref{doublesum} into cases $3j/2 \leq l$, and $3j/2>l$ to see that up to terms exhibiting a stronger gain, it is bounded by
\[
\sum_{k=1}^{\log_2\mu}\left(\|\tilde{\beta}_k^T (\phi \tilde{u}_\mu) \|_{H^1}^2 +
 \left\|P\big(\tilde{\beta}_k^T (\phi \tilde{u}_\mu)\big) \right\|_{L^2}^2\right) ,
\]
after exploiting gains in $[P,\tilde{\beta}_k^N]$ as above.  Exploiting commutators as before, this in turn is bounded by the right hand side of \eqref{ejbound}.

\subsection{Refined local smoothing for the Schr\"odinger equation}\label{ss:schrod}
Here we prove the first half of Theorem \ref{thm:blairrls}, that is, we verify the estimate \eqref{rlsschrod}.  Suppose $f_k$ is spectrally localized to frequencies $\sqrt{\Delta_g} \in (\lambda/2,2\lambda)$ and that $k \in (\lambda/2,2\lambda)$. We observe the following bound for such functions which follows from the bound on the first term in the main estimate from Theorem \ref{thm:intrinsic}:
\begin{equation}\label{quasimodebound}
\max_{1 \leq 2^{-j} \leq \lambda^{\frac 23}} 2^{\frac j4}\|f_k\|_{L^2(S_{<j})} \lesssim \|f_k\|_{L^2( M)} + k^{-1}\|(k^2-\Delta_g) f_k\|_{L^2( M)}.
\end{equation}
Indeed, this follows by simply applying \eqref{intrinsiclocnrg} to $u_\lambda(t,x) = e^{itk}f_k(x)$ and using that $\|\nabla_g f_k\|_{L^2(M)}\lesssim k\|f_k\|_{L^2(M)}$.  It is interesting to note that this shows that if $f_k$ is an $L^2$-normalized quasimodes/spectral clusters satisfying $\|f_k\|=1$ and $k^{-1}\|(k^2-\Delta_g) f_k\|_{L^2( M)}\lesssim 1$, then $\|f_k\|_{L^2(S_{<j})} $ is $O(2^{-j/4})$ for any $1 \leq j \leq \lambda^{\frac 23}$.

To prove \eqref{rlsschrod}, it suffices to prove it for solutions to the homogeneous equation $(D_t+\lambda^{-1}\Delta_g)v_\lambda=0$.  Moreover, it suffices to assume that for some orthogonal collection of eigenfunctions  $\Delta_g \varphi_l =\lambda_l^2\varphi_l$
\[
v_\lambda(t,x) = \sum_{\lambda/2 < \lambda_l < 2\lambda}e^{-\frac{it\lambda_l^2}{\lambda}} \varphi_l(x) \quad \text{with} \quad
\|v_\lambda(0,\cdot)\|_{L^2( M)}^2 = \sum_{\lambda/2 < \lambda_l < 2\lambda} \|\varphi_l\|_{L^2( M)}^2.
\]

Let $\psi$ be a smooth bump function such that $\psi(t)=1$ for $t \in [-1,1]$ and $\psi(t)=0$ for $|t|\geq 2$.  By Plancherel's identity, we have
\begin{equation*}
\|v_\lambda\|_{L^2_T(S_{<j})}^2 \lesssim \int_{S_{<j}} \int_{-\infty}^\infty \Big| \sum_{k\approx \lambda}\sum_{\lambda_l \in[k,k+1)} \widehat{\psi}(\tau+\lambda^{-1}\lambda_l^2)\varphi_l(x)\Big|^2\,d\tau\,dx.
\end{equation*}
Next we observe that since $\widehat{\phi}$ is rapidly decreasing, we may restrict the domain of integration of the inner integral on the right to $-\tau \in [\lambda/4,4\lambda]$.  Indeed, if $\tau$ is outside this interval, then $|\tau+\lambda^{-1}\lambda_l^2| \gtrsim \lambda$, and hence the contribution of $-\tau \notin [\lambda/4,4\lambda]$ to this integral is $O(\lambda^{-N}\|v_\lambda(0,\cdot)\|_{L^2( M)}^2 )$ for any $N$.

We now use that $\sum_{k\approx \lambda}\langle k-\sqrt{-\lambda\tau}\rangle^{-2} \lesssim 1$ for any $\tau$, $\lambda$ to get that
\begin{multline}\label{fttime}
\int_{S_{<j}} \int_{-4\lambda}^{-\frac{\lambda}{4}} \Big| \sum_{k\approx \lambda}\sum_{\lambda_l \in[k,k+1)} \widehat{\psi}(\tau+\lambda^{-1}\lambda_l^2)\varphi_l(x)\Big|^2\,d\tau\,dx \lesssim \\
\sum_{k\approx \lambda} \int_{-4\lambda}^{-\frac{\lambda}{4}} \int_{S_{<j}} \Big| \sum_{\lambda_l \in[k,k+1)} \langle k-\sqrt{-\lambda\tau}\rangle\widehat{\psi}(\tau+\lambda^{-1}\lambda_l^2)\varphi_l(x)\Big|^2\,dx\,d\tau
\end{multline}
When $\lambda_l \in [k,k+1)$, $-\tau \in [\lambda/4,4\lambda]$ we have that
\[
\langle\tau+\lambda^{-1}\lambda_l^2\rangle^{-1} \lesssim \langle\lambda^{-1}(k^2-(-\lambda\tau) )\rangle^{-1} \lesssim \langle k-\sqrt{-\lambda\tau}\rangle^{-1},
\]
and hence the rapid decay of $\widehat{\psi}$ implies that
\[
\max(1,k^{-1}|k^2-\lambda_l^2|) \langle k-\sqrt{-\lambda\tau}\rangle|\widehat{\psi}(\tau+\lambda^{-1}\lambda_l^2)| \lesssim (1+|\tau+\lambda^{-1}\lambda_l^2| )^{-2}.
\]
Thus \eqref{quasimodebound} applied to $f_k = \sum_{\lambda_l \in [k,k+1)} \langle k-\sqrt{-\lambda\tau}\rangle\widehat{\psi}(\tau+\lambda^{-1}\lambda_l^2)\varphi_l$ and orthogonality thus show that the right hand side of \eqref{fttime} is bounded by $\|\varphi_l\|_{L^2( M)}^2$, concluding the proof of \eqref{rlsschrod}.

\begin{remark}
Let $M\subset \RR^n$ be a $C^\infty$ domain exterior to a compact, connected, strictly convex obstacle.  We remark that the bound \eqref{rlsschrod} for solutions to the semiclassical equation can be combined with \eqref{locsmoothunit} to give estimates on solutions to the classical Schr\"odinger equation $(D_t+\Delta)v_\lambda=0$ which are frequency localized in that $v(t,\cdot) = \beta(\lambda^{-2}\Delta)v(t,\cdot)$ for some $\beta \in C_c^\infty (1/2,2)$ using the functional calculus of the Dirichlet or Neumann Laplacian.  Specifically, we claim that the following refinements of the local smoothing bounds hold:
\begin{equation}\label{rlsschrodclassical}
2^{\frac j4}\lambda^{\frac 12}\|v_\lambda\|_{L^2(\RR\times S_{<j,M})} \lesssim \|v_\lambda(0,\cdot)\|_{L^2(M)}
\end{equation}
where this time $S_{<j,M} = \{x \in M: d(x,\prtl M) \leq 2^{-j}\}$ and the implicit constant is independent of $\lambda $ and $1 \leq j \leq \frac 23\log_2\lambda$.  This extends the results of Ivanovici \cite{ivanoballs} for the exterior to a ball in $\RR^n$ to this more general setting, and reflects that for the classical Schr\"odinger equation, propagation speed is proportional to frequency.

To see this, let $\psi \in C^\infty_c(\overline{M})$ be identically one on a neighborhood $\prtl M$.  As a consequence of \eqref{locsmoothunit}, we have if $F_\lambda:=(D_t+\Delta_g)(\psi v_\lambda)$, then
\[
\lambda^{\frac 12}\|\psi v_\lambda\|_{L^2(\RR \times M)} + \lambda^{-\frac 12}\|F_\lambda\|_{L^2(\RR \times M)} \lesssim \|v_\lambda(0,\cdot)\|_{L^2(M)}.
\]
Note that since $v_\lambda$ is spectrally localized to frequencies near $\lambda \gg 1$, this bound is valid for both boundary conditions in \eqref{BCs}.  For $k \in \mathbb{Z}$, let $I_k = [k,k+1]$ and $I_k^* = [k-1,k+2]$.  If we can show that
\[
2^{\frac j4}\lambda^{\frac 12}\|v_\lambda\|_{L^2((\lambda^{-\frac 12}I_k)\times S_{<j,M})} \lesssim \lambda^{\frac 12}\|\psi v_\lambda\|_{L^2((\lambda^{-\frac 12}I_k^*) \times M)} +
\lambda^{-\frac 12}\|F_\lambda\|_{L^2((\lambda^{-\frac 12}I_k^*) \times M)}
\]
over the dilated intervals $\lambda^{-1/2}I_k = [\lambda^{-1/2}k,\lambda^{-1/2}(k+1)]$, then the prior estimate and summation in $k$ will yield \eqref{rlsschrodclassical}.  Setting
\[
\tilde{v}_\lambda(t,x) := v_\lambda(\lambda^{-1}t,x) \qquad \tilde{F}_\lambda(t,x):=(D_t+\lambda^{-1}\Delta)\tilde{v}_\lambda(t,x) = \lambda^{-1}F(\lambda^{-1}t,x)
\]
means that we are reduced to showing that
\[
2^{\frac j4}\|\tilde{v}_\lambda\|_{L^2(I_k\times S_{<j,M})} \lesssim\|\tilde{v}_\lambda\|_{L^2(I_k^*\times M)}  +
\|\tilde{F}_\lambda\|_{L^2(I_k^*\times M)}
\]
But this can be seen as a consequence of \eqref{rlsschrod} above.  Indeed, by taking $R$ large so that $\supp(\psi) \subset (-R,R)^n$ and by identifying sides of the cube here, we may isometrically embed $\supp(\psi)$ into a compact manifold with boundary.  We may now adapt the argument in \eqref{doubletrick} to the present setting to conclude the proof.
\end{remark}

\end{document}